\documentclass[10pt]{amsart}

\newcounter{mnote}
  \let\oldmarginpar\marginpar
    \renewcommand\marginpar[1]{\-\oldmarginpar[\raggedleft\footnotesize #1]%
    {\raggedright\footnotesize #1}}

\usepackage[english]{babel}

\usepackage{amssymb}
\usepackage[leqno]{amsmath}
\usepackage{mathrsfs}
\usepackage{stmaryrd}
\usepackage{chemarrow}
\usepackage[norelsize]{algorithm2e}
\usepackage{enumerate}
\usepackage{graphicx}
\usepackage[all]{xy}
\usepackage{tikz}
\usetikzlibrary{arrows}
\usepackage{multirow}
\usepackage[colorinlistoftodos]{todonotes}
\usepackage[colorlinks=true, allcolors=blue]{hyperref}
\usepackage[hyperpageref]{backref}
\usepackage{enumitem}
\usepackage{booktabs}

\newtheorem{theorem}{Theorem}[section]
\newtheorem{lemma}[theorem]{Lemma}

\newtheorem{example}[theorem]{Example}

\newtheorem{remark}[theorem]{Remark}

\newcommand{\dx}{\,{\rm d}x}
\newcommand{\dd}{\,{\rm d}}

\newcommand{\sign}{\operatorname{sign}}
\newcommand{\curl}{\operatorname{curl}}
\renewcommand{\div}{\operatorname{div}}
\newcommand{\grad}{\operatorname{grad}}
\newcommand{\tr}{\operatorname{tr}}

\newcommand{\skw}{\operatorname{skw}}

\begin{document}
\title[Decoupled finite element methods]{Decoupled finite element methods for a fourth-order exterior differential equation}
\author{Xuewei Cui}%
\address{School of Mathematics, Shanghai University of Finance and Economics, Shanghai 200433, China}%
\email{xueweicui@stu.sufe.edu.cn}%
 \author{Xuehai Huang}%
 \address{Corresponding author. School of Mathematics, Shanghai University of Finance and Economics, Shanghai 200433, China}%
 \email{huang.xuehai@sufe.edu.cn}%

\thanks{This work was supported by the National Natural Science Foundation of China Project 12171300. 
}
\keywords{fourth-order exterior differential equation, decoupled finite element method, Helmholtz decomposition, finite element exterior calculus, conforming finite element method}
\subjclass[2010]{
58J10;   
65N30;   
65N12;   
65N22;   
}

\begin{abstract}
This paper proposes novel decoupled finite element methods for a fourth-order exterior differential equation. Based on differential complexes and the Helmholtz decomposition, the fourth-order exterior differential equation is decomposed into two second-order exterior differential equations and one generalized Stokes equation. A key advantage of this decoupled formulation is that it avoids the use of quotient spaces. A family of conforming finite element methods are developed for the decoupled formulation. Numerical results are provided for verifying the decoupled finite element methods of the biharmonic equation in three dimensions.
\end{abstract}

\maketitle


\section{Introduction}
This paper focuses on exploring decoupled discrete methods for the fourth-order exterior differential equation on a bounded polytope $\Omega \subset \mathbb{R}^{d}$ ($d\geq2$): given $f\in L^2\Lambda^j(\Omega)$ and $g\in L^2\Lambda^{j-1}(\Omega)\cap\delta H_0^{\rm Gd}\Lambda^j(\Omega)$ with $0\leq j\leq d-1$, find $u\in H_0^{\rm Gd}\Lambda^j(\Omega)$ and $\lambda\in H_0\Lambda^{j-1}(\Omega)/\dd H_0\Lambda^{j-2}(\Omega)$ satisfying
\begin{subequations}\label{4thorderpde}
\begin{align}
\label{4thorderpde1}
-\delta\Delta \dd u+\dd\lambda&=f \quad \mathrm{in}\,\,\Omega, \\
\label{4thorderpde2}
\delta u&=g  \quad\hskip0.01in \mathrm{in}\,\,\Omega,
\end{align}
\end{subequations}
where $\dd$ is the exterior derivative, 
$\delta$ is the codifferential operator, $\Delta$ is the Laplacian operator, and 
$$
H_0^{\rm Gd}\Lambda^j(\Omega)=\{\omega\in H_0\Lambda^j(\Omega): \dd\omega\in H_0^1\Lambda^{j+1}(\Omega)\}.
$$
Here, $H_0\Lambda^{j-1}(\Omega)/\dd H_0\Lambda^{j-2}(\Omega)$ denotes the orthogonal complement of $\dd H_0\Lambda^{j-2}(\Omega)$ in $H_0\Lambda^{j-1}(\Omega)$.
The equation \eqref{4thorderpde1} is related to the Hodge decomposition \cite{Arnold2018,ArnoldFalkWinther2006,ArnoldFalkWinther2010}
$$
L^2\Lambda^j(\Omega)= (L^2\Lambda^j(\Omega)\cap\ker(\delta))\oplus\dd H_0\Lambda^{j-1}(\Omega),
$$
where $\ker(\delta)$ means the kernel of the codifferential operator $\delta$.
Applying the codifferential operator $\delta$ to equation \eqref{4thorderpde1} yields
\begin{equation*}
\delta\dd\lambda=\delta f \quad \mathrm{in}\,\,\Omega,
\end{equation*}
which uniquely determines $\lambda$ in the quotient space $H_0\Lambda^{j-1}(\Omega)/\dd H_0\Lambda^{j-2}(\Omega)$. Without loss of generality, we assume $\delta f=0$ in this paper, then $\lambda=0$.

Some examples of problem~\eqref{4thorderpde} are listed as follows.
\begin{itemize}[leftmargin=*]
\item 
For $j=0$, problem \eqref{4thorderpde} becomes the biharmonic equation
\begin{equation}\label{biharmoniceqn}
\Delta^2 u=f\quad \mathrm{in}\,\,\Omega.
\end{equation}
The biharmonic equation finds application in areas such as thin plate theory \cite{Reddy2006} and incompressible fluid flow \cite{Ladyzhenskaya1969}.
\item 
For $j=1$ and $d=3$, problem \eqref{4thorderpde} becomes the quad-curl problem
\begin{subequations}\label{quadcurleqn}
\begin{align}
\label{quadcurleqn1}
-\curl\Delta\curl u +\nabla\lambda&=f \quad \mathrm{in}\,\,\Omega, \\
\label{quadcurleqn2}
\div u&=g  \quad\hskip0.01in \mathrm{in}\,\,\Omega.
\end{align}
\end{subequations}
The quad-curl problem arises in the modeling of magnetized plasmas within magnetohydrodynamics \cite{KingsepChukbarYankov1990,ChaconSimakovZocco2007}.
\item 
For $j=d-1$, problem \eqref{4thorderpde} becomes the fourth-order div problem
\begin{subequations}\label{quaddiveqn}
\begin{align}
\label{quaddiveqn1}
\nabla\Delta\div u+\curl\lambda&=f \quad \mathrm{in}\,\,\Omega, \\
\label{quaddiveqn2}
\curl^*u&=g  \quad\hskip0.01in \mathrm{in}\,\,\Omega.
\end{align}
\end{subequations}
The fourth-order div operator appears in the strain gradient elasticity \cite{Mindlin1964,MindlinEshel1968,ChenHuangHuang2023}.
\end{itemize}
Here, $\curl$ denotes the exterior derivative operator, acting on 1-forms or $(d-2)$-forms depending on the context, and its adjoint is denoted by $\curl^*$.


Since $H_0^{\rm Gd}\Lambda^j$ conforming finite elements typically require high-degree polynomials  and supersmooth degrees of freedom, such as smooth finite elements in~\cite{HuLinWu2024,ChenHuang2024,ChenHuang2024a,Zhang2016,Zhang2009,LaiSchumaker2007,Zenisek1974,Zenisek1970,BrambleZlamal1970,ArgyrisFriedScharpf1968}, $H(\grad\curl)$ conforming finite elements in \cite{ChenHuang2024,ChenHuang2024a,ZhangZhang2020,ZhangWangZhang2019,HuZhangZhang2020} and $H(\grad\div)$ conforming finite elements in \cite{ChenHuang2024,ChenHuang2024a,ZhangZhang2022}, we will consider decoupled finite element methods for the fourth-order exterior differential equation~\eqref{4thorderpde} in this paper.

Applying the framework in \cite{ChenHuang2018} to
the de Rham complex
\begin{equation*}
H^*\Lambda^{j+4}\xrightarrow{\delta} H^*\Lambda^{j+3}\xrightarrow{\delta} L^{2}\Lambda^{j+2}\xrightarrow{\delta} H^{-1}(\delta, \Lambda^{j+1})\xrightarrow{\delta}H^{-1}\Lambda^{j} \cap\ker(\delta)\xrightarrow{}0
\end{equation*}
with $H^{-1}(\delta, \Lambda^{j+1}):=\{\omega\in H^{-1}\Lambda^{j+1}: \delta\omega\in H^{-1}\Lambda^{j}\}$, we derive several Helmholtz decompositions 
\begin{equation*}
(H_{0}\Lambda^{j+1})'=H^{-1}(\delta, \Lambda^{j+1})=\dd H_{0}\Lambda^{j}\oplus
\delta L^{2}\Lambda^{j+2},
\end{equation*}
\begin{equation*}
L^{2}\Lambda^{j+2}=\dd H_{0}^1\Lambda^{j+1} \oplus^{\perp} \delta H^*\Lambda^{j+3}=\dd H_{0}^1\Lambda^{j+1} \oplus^{\perp} \delta H^1\Lambda^{j+3},
\end{equation*}
and decouple problem~\eqref{4thorderpde} into two second-order exterior differential equations and one generalized Stokes equation:
find $u, w\in H_{0}\Lambda^{j}$, $\lambda,z\in H_0\Lambda^{j-1}/\dd H_0\Lambda^{j-2}$,
$\phi\in H_{0}^{1}\Lambda^{j+1}$,
$p\in L^{2}\Lambda^{j+2}$ and $r\in H^*\Lambda^{j+3}/\delta H^*\Lambda^{j+4}$ such that
\begin{subequations}\label{intro:decoupleform}
\begin{align}
(\dd w,\dd v)+(v,\dd\lambda)&=(f,v), \label{intro:decoupledform1}
 \\
(w,\dd\eta)&=0, \label{intro:decoupledform10}
 \\
 (\nabla\phi,\nabla\psi)+(\dd\psi+\delta s, p)&=(\dd w,\psi), \label{intro:decoupledform2}\\
 (\dd \phi+\delta r, q) &=0,\label{intro:decoupledform3}\\
  (\dd u,\dd \chi)+(\chi,\dd z)&=(\phi,\dd\chi), \label{intro:decoupledform4} \\
(u,\dd\mu)&=(g,\mu), \label{intro:decoupledform40}
\end{align}
for any $v,\chi\in H_{0}\Lambda^{j}$, $\eta,\mu\in H_0\Lambda^{j-1}/\dd H_0\Lambda^{j-2}$,
$\psi\in H_{0}^{1}\Lambda^{j+1}$,
$ q\in L^2\Lambda^{j+2}$ and $s\in H^*\Lambda^{j+3}/\delta H^*\Lambda^{j+4}$.
\end{subequations}
The decoupled formulation \eqref{intro:decoupleform} covers many decoupled formulations in literature: the decoupled formulation of the biharmonic equation in two and three dimensions in~\cite{HuangHuangXu2012,Huang2010,ChenHuang2018,Gallistl2017} and the decoupled formulation of the quad-curl problem in three dimensions in \cite[Section 3.4]{ChenHuang2018}.

Conforming finite element methods of the decoupled formulation~\eqref{intro:decoupleform} with $j=0$ and $d=2$, i.e. biharmonic equation in two dimensions, are shown in \cite[Section 4.2]{ChenHuang2018}.
A low-order nonconforming finite element discretization of the decoupled formulation~\eqref{intro:decoupleform} with $j=1$ and $d=3$, i.e. the quad-curl problem in three dimensions, is designed in \cite{CaoChenHuang2022}.
The decoupled formulation of the biharmonic equation in two dimensions is employed to implement the $H^2$-conforming finite element methods using $H^1$-conforming finite elements in \cite{AinsworthParker2024}, and design fast solvers for the Morley element method in \cite{HuangHuangXu2012,FengZhang2016,HuangShiWang2021}.
Based on nonconforming finite element Stokes complexes, nonconforming finite element methods of the quad-curl problem in three dimensions in \cite{Huang2023,HuangZhang2024} are equivalent to nonconforming discretization of the decoupled formulation \eqref{intro:decoupleform}, which is also helpful for developing efficient solvers.
We refer to~\cite{ChenHuang2018,Schedensack2016,Zhang2018,AinsworthParker2024a,BrennerCavanaughSung2024,BrennerSunSung2017} for different decoupled formulations rather than the decoupled formulation \eqref{intro:decoupleform}.
Decoupled formulations and finite element methods for the triharmonic equation have also been developed in \cite{ChenHuang2018,Gallistl2017,Schedensack2016,AnHuangZhang2024}.

While implementing the discretization of the decoupled formulation \eqref{intro:decoupleform} poses challenges due to the quotient spaces $H_0\Lambda^{j-1}/\dd H_0\Lambda^{j-2}$ and $L^2\Lambda^{j+3}/\delta H^*\Lambda^{j+4}$. To circumvent this, 
by using the fact that $r=0$ and $\lambda=z=0$ in \eqref{intro:decoupleform}, we propose the following novel decoupled formulation without quotient spaces:
find $u, w\in H_{0}\Lambda^{j}$, $\lambda,z\in H_0\Lambda^{j-1}$,
$\phi\in H_{0}^{1}\Lambda^{j+1}$,
$ p\in L^{2}\Lambda^{j+2}$ and $r\in H^*\Lambda^{j+3}$ such that
\begin{subequations}\label{intro:decoupleformnew}
\begin{align}
(\dd w,\dd v)+(v,\dd\lambda)&=(f,v)\qquad\;\forall~v\in H_{0}\Lambda^{j}, \label{intro:decoupledformnew1}
 \\
(w,\dd\eta)-(\lambda,\eta)&=0\qquad\qquad\, \forall~\eta\in H_0\Lambda^{j-1}, \label{intro:decoupledformnew10}
 \\
 (\nabla\phi,\nabla\psi)+(r,s)+(\dd\psi+\delta s, p)&=(\dd w,\psi)\quad\;\forall~\psi\in H_{0}^{1}\Lambda^{j+1}, s\in H^*\Lambda^{j+3}, \label{intro:decoupledformnew2}\\
 (\dd \phi+\delta r, q) &=0\qquad\qquad\,\forall~q\in L^2\Lambda^{j+2},\label{intro:decoupledformnew3}\\
  (\dd u,\dd \chi)+(\chi,\dd z)&=(\phi,\dd\chi)\quad\;\,\forall~\chi\in H_{0}\Lambda^{j}, \label{intro:decoupledformnew4} \\
(u,\dd\mu)-(z,\mu)&=(g,\mu)\qquad\;\forall~\mu\in H_0\Lambda^{j-1}. \label{intro:decoupledformnew40}
\end{align}
\end{subequations}
Compared to the decoupled formulation \eqref{intro:decoupleform} and the variational formulation \eqref{4thorderpdeweakform}, the decoupled formulation \eqref{intro:decoupleformnew} avoids quotient spaces entirely, making it significantly more amenable to the design of finite element methods and efficient solvers.

Both \eqref{intro:decoupledformnew1}-\eqref{intro:decoupledformnew10} and \eqref{intro:decoupledformnew4}-\eqref{intro:decoupledformnew40} are second-order exterior differential equations, which can be discretized by finite element differential forms in \cite{Hiptmair2001,Hiptmair2002,Arnold2018,ArnoldFalkWinther2006,ArnoldFalkWinther2010}.
We focus on designing finite element methods for the generalized Stokes equation~\eqref{intro:decoupledformnew2}-\eqref{intro:decoupledformnew3}.
We use $\mathbb{P}_{k}\Lambda^{j+1}(T) + b_{T}\,\delta\mathbb{P}_{k}\Lambda^{j+2}(T)$ as the shape function space to discretize $\phi\in H_{0}^{1}\Lambda^{j+1}$,
trimmed finite element differential form $V_{k,h}^{\delta,-}\Lambda^{j+2}$ to discretize $p\in L^{2}\Lambda^{j+2}$, and $V_{k,h}^{\delta,-}\Lambda^{j+3}$ to discretize $r\in H^*\Lambda^{j+3}$.
The resulting finite element method can be regarded as the generalization of the MINI element method for Stokes equation in \cite{ArnoldBrezziFortin1984,BoffiBrezziFortin2013}.
Error analysis is present for the decoupled finite element method.
It's worth mentioning that
the error estimate for $\|p-p_h\|$ is optimal when $j\leq d-3$, while the error estimate of $\|p-p_h\|$ for the MINI element method of Stokes equation is suboptimal \cite{ArnoldBrezziFortin1984}. 

Thanks to the discrete solutions $r_{h}=0$ and $\lambda_{h}=z_{h}=0$,
we can replace the terms $(\lambda_h,\eta)$, $(r_h,s)$, and $(z_h,\mu)$ in the finite element discretization of the decoupled formulation~\eqref{intro:decoupleformnew} with the weighted discrete $L^2$-inner products $\langle\lambda_h,\eta\rangle_D$, $\langle r_h,s\rangle_D$, and $\langle z_h,\mu\rangle_D$ respectively. Here, 
the matrix representation of these weighted inner products is simply the diagonal matrix of the standard $L^2$ mass matrix.
The resulting decoupled finite element method seeks $w_{h}\in \mathring{V}_{k,h}^{\dd}\Lambda^{j}$,
$\phi_{h}\in\Phi_h$,
$p_h\in V_{k,h}^{\delta,-}\Lambda^{j+2}$, $r_h\in V_{k,h}^{\delta,-}\Lambda^{j+3}$,
$u_h\in \mathring{V}_{k+1,h}^{\dd,-}\Lambda^{j}$ and $\lambda_h,z_h\in \mathring{V}_{k+1,h}^{\dd,-}\Lambda^{j-1}$ such that
\begin{subequations}\label{intro:decoupleMINIdiag}
\begin{align}
(\dd w_h,\dd v)+(v,\dd\lambda_h)&=(f,v)
  &&\!\!\forall \, v\in \mathring{V}_{k,h}^{\dd}\Lambda^{j},\label{intro:decoupleMINIdiag1}\\
(w_h,\dd\eta)-\langle\lambda_h,\eta\rangle_D&=0 &&\!\!\forall \, \eta\in \mathring{V}_{k+1,h}^{\dd,-}\Lambda^{j-1}, \label{intro:decoupleMINIdiag10} \\
(\nabla\phi_{h},\nabla\psi)+\langle r_h,s\rangle_D+(\dd \psi
 +\delta s, p_{h})&=(\dd w_{h},\psi)
  &&\!\!\forall \, \psi\in \Phi_h, s\in V_{k,h}^{\delta,-}\Lambda^{j+3},\label{intro:decoupleMINIdiag2} \\
 (\dd \phi_{h}+\delta r_{h}, q) &=0  &&\!\!\forall \, q\in V_{k,h}^{\delta,-}\Lambda^{j+2},\label{intro:decoupleMINIdiag3}\\
  (\dd u_h,\dd \chi)+(\chi,\dd z_h)&=(\phi_{h},\dd\chi) &&\!\!\forall \, \chi\in \mathring{V}_{k+1,h}^{\dd,-}\Lambda^{j},\label{intro:decoupleMINIdiag4} \\
(u_h,\dd\mu)-\langle z_h,\mu\rangle_D&=(g,\mu) &&\!\!\forall \,  \mu\in \mathring{V}_{k+1,h}^{\dd,-}\Lambda^{j-1}. \label{intro:decoupleMINIdiag40}
\end{align}
\end{subequations}
The diagonal structure of the matrices associated with $\langle\lambda_h,\eta\rangle_D$, $\langle r_h,s\rangle_D$, and $\langle z_h,\mu\rangle_D$ allows for the local elimination of $\lambda_h$, $r_h$, and $z_h$ during the assembly of the linear system for the decoupled mixed method \eqref{intro:decoupleMINIdiag}, reducing the unknowns to $w_h$, $\phi_h$, $p_h$ and $u_h$. After elimination, the coefficient matrices of the mixed methods \eqref{intro:decoupleMINIdiag1}-\eqref{intro:decoupleMINIdiag10} and \eqref{intro:decoupleMINIdiag4}-\eqref{intro:decoupleMINIdiag40} become symmetric and positive definite.

The rest of this paper is organized as follows. Section~\ref{sec:4thede} presents weak formulations of the fourth-order exterior differential equation, and
Section~\ref{sec:decoupleform} focuses on their decoupled variational formulations.
A family of decoupled conforming finite element methods are designed in Section \ref{sec:MINIfem}.
Finally, in Section~\ref{sec:numeresult}, we conduct numerical experiments to validate the theoretical estimates we have developed.

\section{Fourth-order exterior differential equation}\label{sec:4thede}

Weak formulations for the fourth-order exterior differential equation \eqref{4thorderpde} are introduced in this section, and their well-posedness is rigorously established.

\subsection{Notation}
Let $\Omega\subset \mathbb{R}^d$ ($d\geq2$) be a bounded and contractible polytope.
Given integer $m$ and a bounded domain $D \subset \mathbb{R}^{d}$,
let $H^{m}(D)$ be the standard Sobolev space of functions on $D$.
The corresponding norm and
semi-norm are denoted by $\| \cdot \|_{ m,D}$ and $|\cdot|_{m,D}$, respectively. Set $L^{2}(D) = H^{0}(D)$
with  the usual inner product $(\cdot, \cdot)_{D}$. We denote $H_{0}^{m}(D)$ as the closure of $C_{0}^{\infty}(D)$
with respect to the norm $\| \cdot \|_{m,D}$.
In case $D$ is $\Omega$, we abbreviate  $\| \cdot \|_{ m,D}$, $|\cdot|_{m,D}$
and $(\cdot, \cdot)_{D}$  as $\| \cdot \|_{m}$, $|\cdot|_{m}$ and $(\cdot, \cdot)$, respectively.
We also abbreviate $\Vert\cdot\Vert_{0,D}$ and $\Vert\cdot\Vert_{0}$ by $\Vert\cdot\Vert_{D}$ and $\Vert\cdot\Vert$, respectively.
For integer $k\geq0$,
let $\mathbb{P}_{k}(D)$ represent the space of all polynomials in $D$ with the total degree no more than $k$.
Set $\mathbb{P}_{k}(D)=\{0\}$ for $k<0$.
Let $L_{0}^{2}(D)$ be the space of functions in $L^{2}(D)$ with vanishing integral average values.
For a space $B(D)$ defined on $D$,
let $B(D; \mathbb{R}^d):=B(D)\otimes\mathbb{R}^d$ be its vector version. 
Denote by $h_D$ the diameter of $D$. We use $\boldsymbol{n}_{\partial D}$ to denote the unit outward normal vector of $\partial D$, which will be abbreviated as $\boldsymbol{n}$ if not causing any confusion.

We mainly follow the notation set in \cite{Arnold2018,ArnoldFalkWinther2006,ArnoldFalkWinther2010}.
For a $d$-dimensional vector space $V$ and a nonnegative integer $j$,  define the space ${\rm Alt}^jV$ as the space of all skew-symmetric $j$-linear forms. For a multilinear $j$-form $\omega$, its skew-symmetric part
\begin{equation*}
(\skw\omega)(v_1, \ldots, v_j)=\frac{1}{j!}\sum_{\sigma\in\mathfrak{S}_j}\sign(\sigma)\omega(v_{\sigma(1)}, \ldots, v_{\sigma(j)}),\quad v_1, \ldots, v_j\in V
\end{equation*}  
is an alternating form, where $\mathfrak{S}_j$ is the symmetric group of all permutations of the set $\{1,\ldots, j\}$, and $\sign(\sigma)$ denotes the signature of the permutation $\sigma$.
The exterior product or wedge product of $\omega\in {\rm Alt}^iV$ and $\eta\in {\rm Alt}^jV$ is
given by
\begin{equation*}
\omega\wedge\eta={i+j\choose j}\skw(\omega\otimes\eta)\in{\rm Alt}^{i+j}V,
\end{equation*}
where $\otimes$ is the tensor product. It satisfies the anticommutativity law
\begin{equation*}
\omega\wedge\eta=(-1)^{ij}\eta\wedge\omega,\quad \omega\in {\rm Alt}^iV,\; \eta\in {\rm Alt}^jV.
\end{equation*}

An inner product on $V$ induces an inner product on ${\rm Alt}^jV$ as follows
\begin{equation*}
\langle \omega, \eta\rangle=\sum_{\sigma}\omega(e_{\sigma(1)}, \ldots, e_{\sigma(j)})\eta(e_{\sigma(1)}, \ldots, e_{\sigma(j)}),\quad \omega,\eta\in {\rm Alt}^jV,
\end{equation*}
where the sum is over increasing sequences $\sigma: \{1, \ldots, j\}\to\{1, \ldots, d\}$ and $\{e_1, \ldots, e_d\}$ is any orthonormal basis.
If the space $V$ is endowed with an orientation by assigning a positive orientation to some particular ordered basis, 
the volume form $\textsf{vol}\in{\rm Alt}^dV$ is the unique $d$-form characterized by $\textsf{vol}(e_{1}, \ldots, e_{d})=1$ for any positively oriented ordered orthonormal basis $e_{1}, \ldots, e_{d}$.
The Hodge star operator is an isometry of ${\rm Alt}^jV$ onto ${\rm Alt}^{d-j}V$ given by
\begin{equation*}
\omega\wedge\eta=\langle\star\omega, \mu\rangle\textsf{vol},\quad \omega\in {\rm Alt}^jV,\; \eta\in {\rm Alt}^{d-j}V.
\end{equation*}
We have
\begin{equation*}
\star\star\omega=(-1)^{j(d-j)}\omega,\quad \omega\in {\rm Alt}^jV.
\end{equation*}

A differential $j$-form $\omega$ is a section of the $j$-alternating bundle, i.e., a map which assigns to each $x \in D$ an element $\omega_x\in{\rm Alt}^jT_x D$, where $T_x D$ denotes the tangent space to $D$ at $x$.
For a space $B(D)$ defined on $D$, let $B\Lambda^j(D)$ be the space of all $j$-forms whose coefficient function belongs to $B(D)$. Notice that $B\Lambda^0(D)=B(D)$.
The norm $\| \cdot \|_{ m,D}$ and semi-norm $|\cdot|_{m,D}$ are extended to space $H^m\Lambda^j(D)$, and inner product $(\cdot, \cdot)_{D}$ to space $L^2\Lambda^j(D)$. 
Define 
\begin{align*}
H\Lambda^j(D)&:=\{\omega\in L^2\Lambda^j(D): \dd\omega\in L^2\Lambda^{j+1}(D)\},\\
H^{*}\Lambda^j(D)&:=\{\omega\in L^2\Lambda^j(D): \delta\omega\in L^2\Lambda^{j-1}(D)\},
\end{align*}
where $\dd$ is the exterior derivative, and 
the codifferential operator $\delta$ is defined as
\begin{equation}\label{eq:stardelta}
\star\delta\omega=(-1)^j\dd\star\omega.
\end{equation}
We equip the following squared norms for the spaces $H\Lambda^j(D)$ and $H^{*}\Lambda^j(D)$:
\begin{equation*}
\|\omega\|_{H\Lambda^j(D)}^2:=\|\omega\|_{D}^2+\|\dd\omega\|_{D}^2,\quad
\|\omega\|_{H^{*}\Lambda^j(D)}^2:=\|\omega\|_{D}^2+\|\delta\omega\|_{D}^2.
\end{equation*}
When $D=\Omega$, we use the abbreviations $\|\cdot\|_{H\Lambda^j}$ and $\|\cdot\|_{H^{*}\Lambda^j}$ for $\|\cdot\|_{H\Lambda^j(D)}$ and $\|\cdot\|_{H^{*}\Lambda^j(D)}$, respectively.
Denote by $H_0\Lambda^j(D)$ the subspace of $H\Lambda^j(D)$ with vanishing trace.
Define
$$
H_0^{\rm Gd}\Lambda^j(D):=\{\omega\in H_0\Lambda^j(D): \dd\omega\in H_0^1\Lambda^{j+1}(D)\}
$$
with the norm  $(\|\omega\|_D+\|\dd\omega\|_{1,D})^{1/2}$.
We can identify $H_0^m\Lambda^{d}(D)$ and $H_0\Lambda^{d-1}(D)$ as $H_0^m(D)\cap L_0^2(D)$ and $H_0(\div, D)$, respectively.
The spaces $L^2\Lambda^j(D)$ and $H^*\Lambda^j(D)$ are to be interpreted as the trivial space $\{0\}$ for $j\geq d+1$, whereas $\mathbb{P}_{k}\Lambda^{d+1}(D)$ is to be interpreted as $\mathbb R$.
We will abbreviate the Sobolev space $B(D)$ as $B$ when $D=\Omega$ if not causing any confusion.

For sufficient smooth $(j-1)$-form $\omega$ and $j$-form $\mu$,
it holds the integration by parts
\begin{equation}\label{greenidentity}
(\dd\omega, \mu)_D=(\omega, \delta\mu)_D+\int_{\partial D}\tr\omega\wedge\tr(\star\mu).
\end{equation}
For $(d-1)$-dimensional face $F$ of $\partial D$, we have
\begin{equation*}
\int_{F}\tr\omega\wedge\tr(\star\mu)=(-1)^{(d-j)(j-1)}(\tr\omega, \star_F\tr(\star\mu))_F.
\end{equation*}
Let $\kappa$ be the Koszul operator mapping $\mathbb P_{k-1}\Lambda^{d-j+1}(D)$ to $\mathbb P_{k}\Lambda^{d-j}(D)$, then for a contractible domain $D$ it holds 
the decomposition \cite{ArnoldFalkWinther2006,Arnold2018}
\begin{equation}\label{eq:polydecompdelta}
\mathbb{P}_{k}\Lambda^j(D)=\star\kappa\mathbb{P}_{k-1}\Lambda^{d-j+1}(D)\oplus\delta \mathbb{P}_{k+1}\Lambda^{j+1}(D)\quad\textrm{ for } j=1,\ldots, d-1. 
\end{equation}
This implies $\mathbb{P}_{0}\Lambda^j(D)=\delta \mathbb{P}_{1}\Lambda^{j+1}(D)$.
Set $\mathbb P_{k}^-\Lambda^{j}(D):=\dd\mathbb P_{k}\Lambda^{j-1}(D)\oplus\kappa\mathbb P_{k-1}\Lambda^{j+1}(D)$.
We have $\mathbb P_{k}^-\Lambda^{0}(D)=\mathbb P_{k}\Lambda^{0}(D)$ and $\mathbb P_{k}^-\Lambda^{d}(D)=\mathbb P_{k-1}\Lambda^{d}(D)$.
For a $0$-form $v$, $\dd v=\nabla v\cdot\dd\boldsymbol{x}$ with $\dd\boldsymbol{x}=(\dx_1, \ldots, \dx_d)^{\intercal}$. If not causing any confusion, we will use $\nabla v$ to represent $\dd v$ for $0$-form $v$.

Let $\{\mathcal{T}_{h}\}_{h>0}$ be a regular family of simplicial meshes of $\Omega\subset \mathbb{R}^{d}$, where $h=\max_{T\in\mathcal{T}_{h}}h_T$.
For $\ell=0,1,\ldots, d-1$,
denote by $\Delta_{\ell}(\mathcal T_h)$ and $\Delta_{\ell}(\mathring{\mathcal T}_h)$ the set of all subsimplices and all interior subsimplices of dimension $\ell$ in the partition $\mathcal{T}_{h}$, respectively.
For a simplex $T$, we let 
$\Delta_{\ell}(T)$ denote the set of subsimplices of dimension $\ell$. 
For a subsimplex $f$ of $\mathcal{T}_{h}$, let $\mathcal T_f$ be the set of all simplices in $\mathcal{T}_h$ sharing $f$.
Denote by $\omega_{T}$ the union of all the simplices in the set $\{\mathcal T_{\texttt{v}}\}_{\texttt{v}\in\Delta_0(T)}$.
For integer $k\geq0$, define 
$\mathbb P_k(\mathcal T_h):=\mathbb P_k\Lambda^{0}(\mathcal T_h)$ with
\begin{align*}
\mathbb P_k\Lambda^{j}(\mathcal T_h)&:=\{v\in L^2\Lambda^{j}(\Omega): v|_T\in\mathbb P_k\Lambda^{j}(T) \;\textrm{ for all } T\in\mathcal{T}_h\}, \\
\mathbb P_{k}^-\Lambda^{j}(\mathcal T_h)&:=\{v\in L^2\Lambda^{j}(\Omega): v|_T\in\mathbb P_{k}^-\Lambda^{j}(T) \;\textrm{ for all } T\in \mathcal{T}_{h}\}.
\end{align*}
Let $Q_{k,T} : L^{2}\Lambda^j(T)\rightarrow \mathbb P_k\Lambda^j(T)$ for $T\in\mathcal T_h$ be the $ L^2$-orthogonal projection operator.



In this paper, we use ``$\lesssim \cdot\cdot\cdot$'' to mean that ``$\leq C \cdot\cdot\cdot$'',
where $C$ is a generic positive constant independent of $h$,
which may have different values in different forms.
And $A \eqsim B$ equivalents to $A \lesssim B$ and $B \lesssim A$.
For a Hilbert space $V$ with inner product $(\cdot, \cdot)_V$ and a subspace $U\subseteq V$, 
the orthogonal complement of $U$ in $V$, denoted by $V/U$, is defined as
\begin{equation*}
V/U:=\{\omega\in V: (\omega, \mu)_V=0 \;\textrm{ for all } \mu\in U \}.
\end{equation*}

\subsection{Variational formulations}

Let $f\in L^2\Lambda^j\cap\ker(\delta)$ and $g\in L^2\Lambda^{j-1}\cap\delta H_0^{\rm Gd}\Lambda^j$ with $0\leq j\leq d-1$.
The weak formulation of problem~\eqref{4thorderpde} is to
find $u\in H_0^{\rm Gd}\Lambda^j$ and $\lambda\in H_0\Lambda^{j-1}/\dd H_0\Lambda^{j-2}$
 such that
\begin{subequations}\label{4thorderpdeweakform}
\begin{align}
\label{4thorderpdeweakform1}
(\nabla\dd u, \nabla\dd v)+(v, \dd\lambda)&=(f,v) \quad\quad \forall\,\, v\in H_0^{\rm Gd}\Lambda^j, \\
\label{4thorderpdeweakform2}
(u, \dd\mu)&=(g,\mu) \quad\quad \forall\,\, \mu\in H_0\Lambda^{j-1}/\dd H_0\Lambda^{j-2}.
\end{align}
\end{subequations}

\begin{lemma}
The variational formulation \eqref{4thorderpdeweakform} is well-posed. The solution $(u, \lambda)\in H_0^{\rm Gd}\Lambda^j\times (H_0\Lambda^{j-1}/\dd H_0\Lambda^{j-2})$ of problem \eqref{4thorderpdeweakform} satisfies $\lambda=0$ and the stability
\begin{equation}\label{eq:4thorderpdestability0}
\|u\|+\|\dd u\|_1\lesssim \|f\|+\|g\|. 
\end{equation}
\end{lemma}
\begin{proof}
For $u\in H_0^{\rm Gd}\Lambda^j\cap\ker(\delta)$, by the Poincar\'e inequality \cite[(4.7)]{Arnold2018}, we have the coercivity
\begin{equation*}
\|u\|^2+\|\dd u\|_1^2\lesssim |\dd u|_1^2=(\nabla\dd u, \nabla\dd u).
\end{equation*}
For $\lambda\in H_0\Lambda^{j-1}/\dd H_0\Lambda^{j-2}$, applying the Poincar\'e inequality again, it follows the inf-sup condition
\begin{equation*}
\|\lambda\|_{H\Lambda^{j-1}}\lesssim \|\dd \lambda\|=\frac{(\dd\lambda, \dd\lambda)}{\|\dd \lambda\|}\leq \sup_{v\in H_0^{\rm Gd}\Lambda^j}\frac{(v, \dd\lambda)}{\|v\|+\|\dd v\|_1}.
\end{equation*}
Then apply the Babu\v{s}ka-Brezzi theory \cite{BoffiBrezziFortin2013} to acquire the well-posedness of the variational formulation \eqref{4thorderpdeweakform} and the stability
\begin{align}
\notag
&\|u\|+\|\dd u\|_1+\|\lambda\|_{H\Lambda^{j-1}} \\
\label{eq:4thorderpdestability}
&\qquad\qquad\lesssim \sup_{v\in H_0^{\rm Gd}\Lambda^j, \mu\in H_0\Lambda^{j-1}/\dd H_0\Lambda^{j-2}}\frac{(\nabla\dd u, \nabla\dd v)+(v, \dd\lambda)+(u, \dd\mu)}{\|v\|+\|\dd v\|_1+\|\mu\|_{H\Lambda^{j-1}}}.
\end{align}
Hence, \eqref{eq:4thorderpdestability0} is true.

By taking $v=\dd\lambda$ in \eqref{4thorderpdeweakform1}, we get from $\delta f=0$ that $\dd\lambda=0$. Thus, $\lambda=0$.
\end{proof}

For $j=0$, the well-posedness of the variational formulation \eqref{4thorderpdeweakform} of the biharmonic equation~\eqref{biharmoniceqn} is well-known \cite{Ciarlet1978}.
For $j=1$ and $d=3$, the well-posedness of the variational formulation \eqref{4thorderpdeweakform} of the quad-curl problem~\eqref{quadcurleqn} has been established in \cite{Sun2016,SunZhangZhang2019}.
For $j=d-1$, the well-posedness of the variational formulation \eqref{4thorderpdeweakform} of the fourth-order div problem~\eqref{quaddiveqn} is shown in \cite{ZhangZhang2022}.

It's unfavorable to discretize the variational formulation \eqref{4thorderpdeweakform} immediately, due to the quotient space $H_0\Lambda^{j-1}/\dd H_0\Lambda^{j-2}$.
We can rewrite it as the following equivalent formulation without quotient spaces:
find $u\in H_0^{\rm Gd}\Lambda^j$ and $\lambda\in H_0\Lambda^{j-1}$
 such that
\begin{subequations}\label{4thorderpdeweakformag}
\begin{align}
\label{4thorderpdeweakformag1}
(\nabla\dd u, \nabla\dd v)+(v, \dd\lambda)&=(f,v) \quad\quad \forall\,\, v\in H_0^{\rm Gd}\Lambda^j, \\
\label{4thorderpdeweakformag2}
(u, \dd\mu)-(\lambda,\mu)&=(g,\mu) \quad\quad \forall\,\, \mu\in H_0\Lambda^{j-1}.
\end{align}
\end{subequations}

\begin{lemma}\label{lem:4thorderpdeweakformag}
The variational formulation \eqref{4thorderpdeweakformag} is well-posed, and equivalent to the variational formulation \eqref{4thorderpdeweakform} in the sense that the solution $(u, \lambda)\in H_0^{\rm Gd}\Lambda^j\times H_0\Lambda^{j-1}$ of the variational formulation \eqref{4thorderpdeweakformag} 
also satisfies the variational formulation \eqref{4thorderpdeweakform}.
\end{lemma}
\begin{proof}
By the stability \eqref{eq:4thorderpdestability} with $\lambda=0$, we have for $u\in H_0^{\rm Gd}\Lambda^j$ that
\begin{align*}
 \|u\|+\|\dd u\|_1 & \lesssim \sup_{v\in H_0^{\rm Gd}\Lambda^j, \mu\in H_0\Lambda^{j-1}}\frac{(\nabla\dd u, \nabla\dd v)+(u, \dd\mu)}{\|v\|+\|\dd v\|_1+\|\mu\|_{H\Lambda^{j-1}}} \\
&\leq \|\nabla\dd u\|+\sup_{\mu\in H_0\Lambda^{j-1}}\frac{(u, \dd\mu)}{\|\mu\|_{H\Lambda^{j-1}}}.
\end{align*}
Then we have
\begin{equation*}
\|u\|+\|\dd u\|_1 \eqsim \|\nabla\dd u\|+\sup_{\mu\in H_0\Lambda^{j-1}}\frac{(u, \dd\mu)}{\|\mu\|_{H\Lambda^{j-1}}}\qquad\forall~u\in H_0^{\rm Gd}\Lambda^j.
\end{equation*}
It can be verified that
\begin{equation*}
\|\lambda\|_{H\Lambda^{j-1}}\eqsim \|\lambda\|+\sup_{v\in H_0^{\rm Gd}\Lambda^j}\frac{(v, \dd\lambda)}{\|v\|+\|\dd v\|_1}\qquad\forall~\lambda\in H_0\Lambda^{j-1}.
\end{equation*}
By applying the Zulehner theory \cite{Zulehner2011}, we conclude the well-posedness of the variational formulation \eqref{4thorderpdeweakformag} from the last two equivalences.

Assume $(u, \lambda)\in H_0^{\rm Gd}\Lambda^j\times H_0\Lambda^{j-1}$ is the solution of the variational formulation~\eqref{4thorderpdeweakformag}. 
Thanks to $\delta f=0$, choose $v=\dd\lambda$ in \eqref{4thorderpdeweakformag1} to obtain $\dd\lambda=0$. Then $\lambda=\dd\omega$ with $\omega\in H_0\Lambda^{j-2}$.
By choosing $\mu=\lambda$ in \eqref{4thorderpdeweakformag2}, we have $\lambda=0$ from $(g, \lambda)=(g, \dd\omega)=(\delta g,\omega)=0$. Therefore, $(u, \lambda)$ satisfies the variational formulation \eqref{4thorderpdeweakform}.
\end{proof}

\begin{example}\rm
When $j=0$, the variational formulation \eqref{4thorderpdeweakformag} of the biharmonic equation~\eqref{biharmoniceqn} is to find $u\in H_{0}^2(\Omega)$ such that
\begin{align*}
(\nabla^2 u, \nabla^2 v) &=(f,v) \qquad\quad \forall~v\in H_{0}^2(\Omega).
\end{align*}
\end{example}

\begin{example}\rm
When $j=1$ and $d=3$, the variational formulation \eqref{4thorderpdeweakformag} of the quad-curl problem~\eqref{quadcurleqn} is to find $u\in H_{0}(\grad\curl,\Omega)$ and $\lambda\in H_{0}^1(\Omega)$ such that
\begin{align*}
(\nabla\curl u, \nabla\curl v)+(v, \nabla\lambda)&=(f,v) \quad\quad \forall\,\, v\in H_{0}(\grad\curl,\Omega), \\
(u, \nabla\mu)-(\lambda,\mu)&=(g,\mu) \quad\quad \forall\,\, \mu\in H_0^1(\Omega),
\end{align*}
where $H_{0}(\grad\curl,\Omega):= H_0^{\rm Gd}\Lambda^1$.
\end{example}

\begin{example}\rm
When $j=d-1$, the variational formulation \eqref{4thorderpdeweakformag} of the fourth-order div problem~\eqref{quaddiveqn} is to find $u\in H_{0}(\grad\div,\Omega)$ and $\lambda\in H_0(\curl,\Omega)$ such that
\begin{align*}
(\nabla\div u, \nabla\div v)+(v, \curl\lambda)&=(f,v) \quad\quad \forall\,\, v\in H_{0}(\grad\div,\Omega), \\
(u, \curl\mu)-(\lambda,\mu)&=(g,\mu) \quad\quad \forall\,\, \mu\in H_0(\curl,\Omega),
\end{align*}
where $H_{0}(\grad\div,\Omega):= H_0^{\rm Gd}\Lambda^{d-1}$ and $H_0(\curl,\Omega)=H_0\Lambda^{d-2}$.
\end{example}

\begin{remark}\rm
When $\delta f$ is not zero, the weak formulation \eqref{4thorderpdeweakform} of problem~\eqref{4thorderpde} is equivalent to
find $u\in H_0^{\rm Gd}\Lambda^j$, $\lambda,\lambda_0\in H_0\Lambda^{j-1}$ and $r\in H_0\Lambda^{j-2}$
 such that
\begin{subequations}\label{4thorderpdeweakformgeneral}
\begin{align}
\label{4thorderpdeweakformgeneral1}
(\dd \lambda, \dd \mu)+(\mu, \dd r )&=(f,\dd\mu) \quad\qquad\;\;\, \forall\,\, \mu\in H_0\Lambda^{j-1}, \\
\label{4thorderpdeweakformgeneral2}
(\lambda, \dd s)-(r,s)&=0 \qquad\qquad\qquad\; \forall\,\, s\in H_0\Lambda^{j-2},\\
\label{4thorderpdeweakformgeneral3}
(\nabla\dd u, \nabla\dd v)+(v, \dd\lambda_0)&=(f-\dd \lambda,v) \quad\quad \forall\,\, v\in H_0^{\rm Gd}\Lambda^j, \\
\label{4thorderpdeweakformgeneral4}
(u, \dd\mu_0)-(\lambda_0,\mu_0)&=(g,\mu_0) \qquad\qquad\, \forall\,\, \mu_0\in H_0\Lambda^{j-1}.
\end{align}
\end{subequations}
In the variational formulation \eqref{4thorderpdeweakformgeneral}, the variables $ u $ and $ \lambda $ are decoupled. Specifically, we first solve equations \eqref{4thorderpdeweakformgeneral1}-\eqref{4thorderpdeweakformgeneral2} to determine $\lambda$, and then solve equations \eqref{4thorderpdeweakformgeneral3}-\eqref{4thorderpdeweakformgeneral4} to obtain $u$. Note that this formulation does not involve any quotient spaces, and the solutions satisfy $r = 0$ and $\lambda_0 = 0$.
\end{remark}

\section{Decoupled variational formulations}\label{sec:decoupleform}

We shall develop two decoupled variational formulations for the variational formulation \eqref{4thorderpdeweakform} in this section: decouple the fourth-order exterior differential equation~\eqref{4thorderpde} into two second-order exterior differential equations and one generalized Stokes equation.

\subsection{Decoupled formulation with quotient spaces}

We will apply the framework in~\cite{ChenHuang2018} to decouple the variational formulation \eqref{4thorderpdeweakform} into lower order differential equations.

Recall the de Rham complex \cite{CostabelMcIntosh2010,ArnoldFalkWinther2006,Arnold2018}
\begin{equation}\label{eq:deRham}
H^*\Lambda^{j+4}\xrightarrow{\delta} H^*\Lambda^{j+3}\xrightarrow{\delta} L^{2}\Lambda^{j+2}\xrightarrow{\delta} H^{-1}\Lambda^{j+1} \xrightarrow{\delta}H^{-2}\Lambda^{j}\cap\ker(\delta) \xrightarrow{}0, 
\end{equation}
which is exact on the contractible domain $\Omega$. 
Applying the tilde operation in \cite{ChenHuang2022b},
the de Rham complex \eqref{eq:deRham} implies the exact de Rham complex
\begin{equation}\label{eq:deRham1}
H^*\Lambda^{j+4}\xrightarrow{\delta} H^*\Lambda^{j+3}\xrightarrow{\delta} L^{2}\Lambda^{j+2}\xrightarrow{\delta} H^{-1}(\delta, \Lambda^{j+1})\xrightarrow{\delta}H^{-1}\Lambda^{j} \cap\ker(\delta)\xrightarrow{}0,
\end{equation}
where $H^{-1}(\delta, \Lambda^{j+1}):=\{\omega\in H^{-1}\Lambda^{j+1}: \delta\omega\in H^{-1}\Lambda^{j}\}$. We understand $H^{-1}(\delta, \Lambda^{d})=L_0^2(\Omega)$.
With the de Rham complex~\eqref{eq:deRham1}, 
we build up the following commutative
diagram
\begin{equation}\label{cd}
		\begin{array}{c}
			\xymatrix{
 H_{0}^{1}\Lambda^{j+1} \ar[r]^-{\Delta} &  H^{-1}\Lambda^{j+1}  \\
L^{2}\Lambda^{j+2}/\delta H^*\Lambda^{j+3} \ar[r]^-{\delta } &  H^{-1}(\delta, \Lambda^{j+1})\ar@{}[u]|{\bigcup}
\ar[r]^-{\delta } & H^{-1}\Lambda^{j} \cap\ker(\delta) \ar[r]^-{} & 0. \\
& \ar[u]^{I} H_{0}\Lambda^{j+1} & \ar[l]_-{\dd} H_{0}\Lambda^{j}/\dd H_{0}\Lambda^{j-1}
\ar[u]_{\delta\dd}
}
\end{array}
\end{equation}

To derive Helmholtz decompositions, we first present the existence of regular potentials. 
\begin{lemma}
Let $0\leq j\leq d-1$.
We have
\begin{equation}\label{eq:dH1}
\dd H_{0}\Lambda^{j}=\dd H_{0}^1\Lambda^{j},\quad\delta H^*\Lambda^{j}=\delta H^1\Lambda^{j}.\end{equation}
\end{lemma}
\begin{proof}
By the de Rham complex \cite{CostabelMcIntosh2010,ArnoldFalkWinther2006,Arnold2018}, it holds that
\begin{equation*}
\dd H_{0}\Lambda^{j}= H_{0}\Lambda^{j+1}\cap\ker(\dd).
\end{equation*}
By Theorem 1.1 in \cite{CostabelMcIntosh2010}, $H_{0}\Lambda^{j+1}\cap\ker(\dd)=\dd H_{0}^1\Lambda^{j}$. Hence,
\begin{equation*}
\dd H_{0}\Lambda^{j}=\dd H_{0}^1\Lambda^{j}.
\end{equation*}
The proof of $\delta H^*\Lambda^{j}=\delta H^1\Lambda^{j}$ is similar.
\end{proof}

The regular potentials in \eqref{eq:dH1} are related to the regular decompositions in \cite{HiptmairXu2007,Hiptmair2002,PasciakZhao2002}.

\begin{lemma}
Let $0\leq j\leq d-1$.
We have 
\begin{equation}\label{eq:Hdualcharac}
(H_{0}\Lambda^{j}/\dd H_{0}\Lambda^{j-1})' = H^{-1}\Lambda^{j} \cap\ker(\delta),
\end{equation}
and Helmholtz decompositions
\begin{align}\label{helmholtzdecomp}
& (H_{0}\Lambda^{j+1})'=H^{-1}(\delta, \Lambda^{j+1})=\dd H_{0}\Lambda^{j}\oplus
\delta L^{2}\Lambda^{j+2}, \\
\label{eq:2formHelmholtzdecomp1}
&\qquad\quad L^{2}\Lambda^{j+2}=\dd H_{0}^1\Lambda^{j+1} \oplus^{\perp} \delta H^*\Lambda^{j+3}, \\
\label{eq:2formHelmholtzdecomp2}
&\qquad\quad L^{2}\Lambda^{j+2}=\dd H_{0}^1\Lambda^{j+1} \oplus^{\perp} \delta H^1\Lambda^{j+3},
\end{align}
where $\oplus^{\perp}$ means the $L^2$ orthogonal direct sum.
\end{lemma}
\begin{proof}
Due to the Poincar\'e inequality, $(\dd\cdot, \dd\cdot)$ is an inner product on the quotient space  $H_{0}\Lambda^{j}/\dd H_{0}\Lambda^{j-1}$, by Riesz-Fr\'echet representation theorem \cite[Theorem 5.5]{Brezis2011}, which implies $\delta\dd: H_{0}\Lambda^{j}/\dd H_{0}\Lambda^{j-1}\to (H_{0}\Lambda^{j}/\dd H_{0}\Lambda^{j-1})'$ is isomorphic. Then
\begin{equation*}
(H_{0}\Lambda^{j}/\dd H_{0}\Lambda^{j-1})'=\delta\dd(H_{0}\Lambda^{j}/\dd H_{0}\Lambda^{j-1})=\delta\dd H_{0}\Lambda^{j}.
\end{equation*}
Recall the Hodge decomposition \cite{ArnoldFalkWinther2006,Arnold2018}
\begin{equation}\label{eq:hodgedecomp}
L^{2}\Lambda^{j+1}=\dd H_{0}\Lambda^{j} \oplus^{\perp} \delta H^*\Lambda^{j+2}.
\end{equation}
Hence it follows
\begin{equation*}
(H_{0}\Lambda^{j}/\dd H_{0}\Lambda^{j-1})'=\delta L^{2}\Lambda^{j+1},
\end{equation*}
which together with complex \eqref{eq:deRham} indicates \eqref{eq:Hdualcharac}.

Thanks to \eqref{eq:Hdualcharac}, applying the framework in \cite{ChenHuang2018} to the diagram \eqref{cd}, we get the Helmholtz decomposition \eqref{helmholtzdecomp}.

By \eqref{eq:dH1}, 
\begin{equation*}
\dd H_{0}\Lambda^{j+1}=\dd H_{0}^1\Lambda^{j+1},\quad \delta H^*\Lambda^{j+3}=\delta H^1\Lambda^{j+3}.
\end{equation*}
We conclude decomposition \eqref{eq:2formHelmholtzdecomp1} from \eqref{eq:hodgedecomp}, and decomposition \eqref{eq:2formHelmholtzdecomp2} from \eqref{eq:2formHelmholtzdecomp1}.
\end{proof}

\begin{remark}\rm
When $j=d-1$, the Helmholtz decompositions \eqref{eq:2formHelmholtzdecomp1}-\eqref{eq:2formHelmholtzdecomp2} become
\begin{equation*}
\{0\}=\int_{\Omega}(H_0^1(\Omega)\cap L_0^2(\Omega))\dx.
\end{equation*}
\end{remark}

Let $f\in L^2\Lambda^j\cap\ker(\delta)$ and $g\in L^2\Lambda^{j-1}\cap\delta H_0^{\rm Gd}\Lambda^j$ with $0\leq j\leq d-1$.
A mixed formulation based on the commutative diagram \eqref{cd} is to find $(\gamma, u)\in  H^{-1}(\delta, \Lambda^{j+1})\times H_{0}\Lambda^{j}$
such that $\delta u=g$ and
\begin{align*}
(\gamma,\beta)_{-1}-\langle\delta \beta,u\rangle&=0  &&\forall \,\,\beta\in H^{-1}(\delta, \Lambda^{j+1}),\\
\langle\delta \gamma,v\rangle&=(f,v)  &&\forall \,\,v\in H_{0}\Lambda^{j}/\dd H_{0}\Lambda^{j-1},
\end{align*}
where $(\gamma,\beta)_{-1}:
=-\langle\Delta^{-1}\gamma,\beta\rangle_{H_{0}^{1}\Lambda^{j+1}\times H^{-1}\Lambda^{j+1}}=
-\langle\gamma,\Delta^{-1}\beta\rangle_{H^{-1}\Lambda^{j+1}\times H_{0}^{1}\Lambda^{j+1}}$.
By introducing the new variable $\phi=-\Delta^{-1}\gamma\in
H_{0}^{1}\Lambda^{j+1}$, the unfolded formulation is
to find $(\gamma,u,\phi)\in  H^{-1}(\delta, \Lambda^{j+1})\times H_{0}\Lambda^{j}\times H_{0}^{1}\Lambda^{j+1}$
such that $\delta u=g$ and
\begin{subequations}\label{unfoldedform}
\begin{align}
(\nabla\phi,\nabla\psi)+\langle \gamma,\dd v-\psi\rangle&=(f,v)
 &&\forall \,\, (\upsilon,\psi)\in (H_{0}\Lambda^{j}/\dd H_{0}\Lambda^{j-1})\times H_{0}^{1}\Lambda^{j+1},
 \label{unfoldedform1}\\
 \langle\beta,\dd u-\phi\rangle&=0
 &&\forall\,\, \beta\in H^{-1}(\delta, \Lambda^{j+1}).\label{unfoldedform2}
\end{align}
\end{subequations}

According to the Helmholtz decomposition (\ref{helmholtzdecomp}), we can set $\gamma=\dd w-\delta p$ and $\beta=\dd\chi-\delta q$ with $w, \chi\in H_{0}\Lambda^{j}/\dd H_{0}\Lambda^{j-1}$ and $ p,  q\in L^{2}\Lambda^{j+2}$.
Then the formulation~\eqref{unfoldedform} is decomposed as follows \cite[Section 3.2]{ChenHuang2018}: find $u\in H_{0}\Lambda^{j}$, $w\in H_{0}\Lambda^{j}/\dd H_{0}\Lambda^{j-1}$,
$\phi\in H_{0}^{1}\Lambda^{j+1}$ and
$ p\in L^{2}\Lambda^{j+2}/\delta H^*\Lambda^{j+3}$ such that $\delta u=g$ and
\begin{align*}
(\dd w,\dd v)&=(f,v)
  &&\forall \,\, v\in H_{0}\Lambda^{j}/\dd H_{0}\Lambda^{j-1},
 \\
 (\nabla\phi,\nabla\psi)+(\dd \psi, p)&=(\dd w,\psi)
  &&\forall \,\, \psi\in H_{0}^{1}\Lambda^{j+1},
 \\
 (\dd \phi, q) &=0  &&\forall \,\,  q \in L^{2}\Lambda^{j+2}/\delta H^*\Lambda^{j+3},\\
  (\dd u,\dd \chi)&=(\phi,\dd\chi) &&\forall \,\, \chi\in H_{0}\Lambda^{j}/\dd H_{0}\Lambda^{j-1}.
\end{align*}

We further employ the Lagrange multiplier to deal with the constraint in the quotient spaces $H_{0}\Lambda^{j}/\dd H_{0}\Lambda^{j-1}$ and 
$L^{2}\Lambda^{j+2}/\delta H^*\Lambda^{j+3}$, which induces the following variational formulation: find $u, w\in H_{0}\Lambda^{j}$, $\lambda,z\in H_0\Lambda^{j-1}/\dd H_0\Lambda^{j-2}$,
$\phi\in H_{0}^{1}\Lambda^{j+1}$,
$ p\in L^{2}\Lambda^{j+2}$ and $r\in H^*\Lambda^{j+3}/\delta H^*\Lambda^{j+4}$ such that
\begin{subequations}\label{decoupleform}
\begin{align}
(\dd w,\dd v)+(v,\dd\lambda)&=(f,v), \label{decoupledform1}
 \\
(w,\dd\eta)&=0, \label{decoupledform10}
 \\
 (\nabla\phi,\nabla\psi)+(\dd\psi+\delta s, p)&=(\dd w,\psi), \label{decoupledform2}\\
 (\dd \phi+\delta r, q) &=0,\label{decoupledform3}\\
  (\dd u,\dd \chi)+(\chi,\dd z)&=(\phi,\dd\chi), \label{decoupledform4} \\
(u,\dd\mu)&=(g,\mu), \label{decoupledform40}
\end{align}
for any $v,\chi\in H_{0}\Lambda^{j}$, $\eta,\mu\in H_0\Lambda^{j-1}/\dd H_0\Lambda^{j-2}$,
$\psi\in H_{0}^{1}\Lambda^{j+1}$,
$ q\in L^2\Lambda^{j+2}$ and $s\in H^*\Lambda^{j+3}/\delta H^*\Lambda^{j+4}$.
\end{subequations}


The  well-posedness of problem \eqref{decoupledform2}-\eqref{decoupledform3} is related to complex \eqref{graddcomplex}.
\begin{lemma}
Let $0\leq j\leq d-2$. The complex
\begin{equation}\label{graddcomplex}
H_0^{\rm Gd}\Lambda^j\times H^*\Lambda^{j+4}\xrightarrow{\begin{pmatrix}\dd & 0\\
0&\delta\end{pmatrix}}H_{0}^{1}\Lambda^{j+1}\times H^*\Lambda^{j+3} \xrightarrow{(\dd, \delta)} L^{2}\Lambda^{j+2} \to 0
\end{equation}
is exact.
\end{lemma}
\begin{proof}
Due to decomposition \eqref{eq:2formHelmholtzdecomp1}, we conclude the exactness of complex \eqref{graddcomplex} from $H_{0}^1\Lambda^{j+1}\cap\ker(\dd)=\dd H_0^{\rm Gd}\Lambda^j$ and $H^*\Lambda^{j+3}\cap\ker(\delta)=\delta H^*\Lambda^{j+4}$.
\end{proof}

Indeed, complex \eqref{graddcomplex} is a distinct representation of the following complex
\begin{equation}\label{H2derhamcomplex1}
H_0\Lambda^{j-1}\xrightarrow{\dd} H_0^{\rm Gd}\Lambda^j\xrightarrow{\dd}H_{0}^{1}\Lambda^{j+1} \xrightarrow{\dd} L^{2}\Lambda^{j+2}/\delta H^*\Lambda^{j+3}\to 0.
\end{equation}
Thanks to complex \eqref{graddcomplex}, we can regard $(\dd, \delta)$ as a generalized divergence operator, and problem \eqref{decoupledform2}-\eqref{decoupledform3} as a generalized Stokes equation.

Next we analyze the well-posedness of the decoupled formulation \eqref{decoupleform}, and show its equivalence to the variational formulation \eqref{4thorderpdeweakform}.
The well-posedness of problem~\eqref{decoupledform1}-\eqref{decoupledform10} and problem~\eqref{decoupledform4}-\eqref{decoupledform40} follows from the Poincar\'e inequality \cite{Arnold2018,ArnoldFalkWinther2006}.
Then we focus on the well-posedness of the generalized Stokes equation \eqref{decoupledform2}-\eqref{decoupledform3}. 

Clearly, we have
\begin{equation*}
(\nabla\phi,\nabla\psi)\leq \|\phi\|_1\|\psi\|_1 \quad \forall~\phi, \psi\in H^{1}\Lambda^{j+1}, 
\end{equation*}
\begin{equation*}
(\dd\psi+\delta s, q)\leq (\|\dd\psi\|+\|\delta s\|)\| q\|\quad \forall~\psi\in H^{1}\Lambda^{j+1}, s\in H^*\Lambda^{j+3},  q\in L^{2}\Lambda^{j+2}.
\end{equation*}

\begin{lemma}
Let $0\leq j\leq d-2$.
For $(\psi, s)\in  H_{0}^{1}\Lambda^{j+1}\times(H^*\Lambda^{j+3}/\delta H^*\Lambda^{j+4})$ satisfying $(\dd\psi+\delta s, q)=0$ for all $ q\in L^{2}\Lambda^{j+2}$, we have the coercivity
\begin{align}\label{AEO}
\|\psi\|_{1}^{2}+\|s\|_{H^*\Lambda^{j+3}}^{2}\lesssim |\psi|_{1}^{2}.
\end{align}
\end{lemma}
\begin{proof}
Notice that
$$(\dd  \psi, q)
 +(\delta s, q)=0 \quad \forall \, q\in  L^{2}\Lambda^{j+2}.$$
Take $ q=\delta s$ to get $s=0$.
Thus, \eqref{AEO} follows from the Poincar\'e inequality for space $H_0^1(\Omega)$ (cf. \cite{Brenner2003}).
\end{proof}

\begin{lemma}
Let $0\leq j\leq d-2$.
For $ q\in L^{2}\Lambda^{j+2}$, it holds the inf-sup condition
\begin{align}\label{AH}
\| q\|\lesssim \sup _{\psi\in  H_{0}^{1}\Lambda^{j+1},\,
s\in H^*\Lambda^{j+3}/\delta H^*\Lambda^{j+4}}
\frac{(\dd\psi+\delta s, q)}{\|\psi\|_{1}+\|s\|_{H^*\Lambda^{j+3}}}.
\end{align}
\end{lemma}
\begin{proof}
By the Helmholtz decomposition \eqref{eq:2formHelmholtzdecomp1},
there exist $\psi\in  H_{0}^{1}\Lambda^{j+1}$ and $s \in H^*\Lambda^{j+3}/\delta H^*\Lambda^{j+4}$ such that
$$
 q=\dd \psi+\delta s,\quad \|\psi\|_{1}+\|s\|_{H^*\Lambda^{j+3}}\lesssim\| q\|.
$$
So
$
(\dd \psi+\delta s, q)
=\| q\|^{2}.$
This means
$$\| q\|(\|\psi\|_{1}+\|s\|_{H^*\Lambda^{j+3}})\lesssim\| q\|^{2}=(\dd \psi+\delta s, q).$$
Then \eqref{AH} follows.
\end{proof}


\begin{theorem}\label{thm:unisoldecoupledform}
Let $0\leq j\leq d-1$.
The generalized Stokes equation \eqref{decoupledform2}-\eqref{decoupledform3} is well-posed.
The mixed formulation \eqref{decoupleform} and
the variational formulation \eqref{4thorderpdeweakform} are equivalent. That is,
if $(w,\lambda)\in H_{0}\Lambda^{j}\times(H_0\Lambda^{j-1}/\dd H_0\Lambda^{j-2})$ is the solution of problem \eqref{decoupledform1}-\eqref{decoupledform10}, $(\phi,  p, r)\in H_{0}^{1}\Lambda^{j+1}\times L^{2}\Lambda^{j+2}\times (H^*\Lambda^{j+3}/\delta H^*\Lambda^{j+4})$ is the solution of problem \eqref{decoupledform2}-\eqref{decoupledform3}, and $(u,z)\in H_{0}\Lambda^{j}\times(H_0\Lambda^{j-1}/\dd H_0\Lambda^{j-2})$ is the solution of problem \eqref{decoupledform4}-\eqref{decoupledform40},
then $r=0$, $\lambda=z=0$, $p\in \dd H_0^1\Lambda^{j+1}$, $\phi=\dd u$, and $(u,\lambda)\in H_0^{\rm Gd}\Lambda^j\times(H_0\Lambda^{j-1}/\dd H_0\Lambda^{j-2})$ satisfies the variational formulation \eqref{4thorderpdeweakform}.
\end{theorem}
\begin{proof}
When $j=d-1$, the generalized Stokes equation \eqref{decoupledform2}-\eqref{decoupledform3} reduces to the following Poisson equation: find $\phi\in H_{0}^{1}(\Omega)\cap L_0^2(\Omega)$ such that
\begin{equation*}
(\nabla\phi,\nabla\psi)=(\div w,\psi) \qquad \forall~\psi\in H_{0}^{1}(\Omega)\cap L_0^2(\Omega),
\end{equation*}
which is obviously well-posed.
For $0\leq j\leq d-2$, by applying the Babu\v{s}ka-Brezzi theory \cite{BoffiBrezziFortin2013},  the well-posedness of the generalized Stokes equation \eqref{decoupledform2}-\eqref{decoupledform3} follows from the coercivity \eqref{AEO} and the inf-sup condition \eqref{AH}. 

Next we prove the equivalence.
Take $v=\dd\lambda$ in (\ref{decoupledform1}) to get $\lambda=0$, $q=\delta r$ in (\ref{decoupledform3}) to get $r=0$ and $\dd\phi=0$, and $\chi=\dd z$ in (\ref{decoupledform4}) to get $z=0$.
So $\phi\in\dd H_0^{\rm Gd}\Lambda^j(\Omega)$.
By \eqref{decoupledform4}-\eqref{decoupledform40}, $\phi=\dd u$ and $u\in H_0^{\rm Gd}\Lambda^j(\Omega)$.
By taking $s=0$ and $\psi=\dd v$ with $v \in H_0^{\rm Gd}\Lambda^j$ in (\ref{decoupledform2}), we acquire
$$ (\nabla\dd u, \nabla\dd v)=(\dd w,\dd v)\quad\forall\,v \in H_0^{\rm Gd}\Lambda^j.$$
This together with (\ref{decoupledform1}) means that $(u,\lambda)\in H_0^{\rm Gd}\Lambda^j\times H_0\Lambda^{j-1}$ satisfies \eqref{4thorderpdeweakform}.

Choose $\psi=0$ in \eqref{decoupledform2} to derive $ p\perp \delta H^*\Lambda^{j+3}$. Thus, $ p\in \dd H_0^1\Lambda^{j+1}$ follows from the Helmholtz decomposition \eqref{eq:2formHelmholtzdecomp1}.
\end{proof}

Therefore, the weak formulation (\ref{4thorderpdeweakform}) of the fourth-order problem \eqref{4thorderpde} is decoupled into two second-order exterior differential equations \eqref{decoupledform1}-\eqref{decoupledform10}, \eqref{decoupledform4}-\eqref{decoupledform40} and one generalized Stokes equation (\ref{decoupledform2})-(\ref{decoupledform3}).

\begin{example}\rm
Taking $j=0$, we get the decoupling of the biharmonic equation~\eqref{biharmoniceqn}: find $u, w\in H_{0}^1(\Omega)$,
$\phi\in H_{0}^{1}\Lambda^{1}$,
$ p\in L^{2}\Lambda^{2}$ and $r\in H^*\Lambda^{3}/\delta H^*\Lambda^{4}$ such that
\begin{align*}
(\nabla w, \nabla v) &=(f,v) \qquad\quad \forall~v\in H_{0}^1(\Omega),
 \\
 (\nabla\phi,\nabla\psi)+(\curl\psi+\delta s, p)&=(\nabla w,\psi) \quad\;\;\; \forall~\psi\in H_{0}^{1}\Lambda^{1}, s\in H^*\Lambda^{3}/\delta H^*\Lambda^{4},\\
 (\curl \phi+\delta r, q) &=0 \qquad\qquad\;\;\, \forall~q\in L^2\Lambda^{2}, \\
  (\nabla u, \nabla \chi)&=(\phi,\nabla\chi) \qquad \forall~\chi\in H_{0}^1(\Omega).
\end{align*}
Such a decoupling of the biharmonic equation in two and three dimensions can be found in~\cite{HuangHuangXu2012,Huang2010,ChenHuang2018,Gallistl2017}, which is employed to design fast solvers for the Morley element method in \cite{HuangHuangXu2012,FengZhang2016,HuangShiWang2021}.
We refer to \cite{ChenHuang2018,Schedensack2016,Zhang2018,AinsworthParker2024a} for different decompositions.
\end{example}

\begin{example}\rm
Taking $j=1$ and $d=3$, we get the decoupling of the quad-curl problem \eqref{quadcurleqn}: find $u, w\in H_{0}(\curl,\Omega)$, $\lambda,z\in H_0^1(\Omega)$,
$\phi\in H_{0}^{1}(\Omega;\mathbb R^3)$ and
$ p\in L_0^{2}(\Omega)$ such that
\begin{align*}
(\curl w, \curl v)+(v, \nabla\lambda)&=(f,v) \qquad\qquad \forall~v\in H_{0}(\curl,\Omega), 
 \\
(w, \nabla\eta)&=0 \qquad\qquad\quad\;\;\, \forall~\eta\in H_{0}^1(\Omega), 
 \\
 (\nabla\phi,\nabla\psi)+(\div\psi, p)&=(\curl w,\psi) \quad\;\;\; \forall~\psi\in H_{0}^{1}(\Omega;\mathbb R^3), \\
 (\div \phi, q) &=0 \qquad\qquad\quad\;\;\, \forall~q\in L_0^2(\Omega),\\
  (\curl u,\curl \chi)+(\chi,\nabla z)&=(\phi, \curl\chi) \qquad \forall~\chi\in H_{0}(\curl,\Omega),  \\
(u,\nabla\mu)&=(g,\mu) \qquad\qquad \forall~\mu\in H_{0}^1(\Omega).
\end{align*}
This is the decoupling in \cite[Section 3.4]{ChenHuang2018}.
We refer to \cite{BrennerCavanaughSung2024,BrennerSunSung2017} for different decompositions.
\end{example}

\begin{example}\rm
Taking $j=d-1$, we get the decoupling of the fourth-order div problem \eqref{quaddiveqn}: 
find $u, w\in H_{0}(\div,\Omega)$, $\lambda,z\in H_0(\curl,\Omega)/\dd H_0\Lambda^{d-3}$
and $\phi\in H_{0}^{1}(\Omega)\cap L_0^2(\Omega)$ such that
\begin{align*}
(\div w,\div v)+(v,\curl\lambda)&=(f,v) \qquad\qquad \forall~v\in H_{0}(\div,\Omega), 
 \\
(w,\curl\eta)&=0 \qquad\qquad\quad\;\;\, \forall~\eta\in H_0(\curl,\Omega)/\dd H_0\Lambda^{d-3}, 
 \\
 (\nabla\phi,\nabla\psi)&=(\div w,\psi) \qquad \forall~\psi\in H_{0}^{1}(\Omega)\cap L_0^2(\Omega), \\
  (\div u,\div \chi)+(\chi,\curl z)&=(\phi, \div\chi) \qquad\, \forall~\chi\in H_{0}(\div,\Omega),  \\
(u,\curl\mu)&=(g,\mu) \qquad\quad\;\;\; \forall~\mu\in H_0(\curl,\Omega)/\dd H_0\Lambda^{d-3}. 
\end{align*}
This decoupled formulation is novel and differs from the one presented in \cite[Section 5]{FanLiuZhang2019}.
\end{example}


\begin{remark}\rm
By applying the integration by parts to $(\delta s, p)$ and $(\delta r, q)$, 
the formulation \eqref{decoupleform} is equivalent to
find $u, w\in H_{0}\Lambda^{j}$, $\lambda,z\in H_0\Lambda^{j-1}/\dd H_0\Lambda^{j-2}$,
$\phi\in H_{0}^{1}\Lambda^{j+1}$,
$ p\in H_0\Lambda^{j+2}$ and $r\in L^2\Lambda^{j+3}/\delta H^*\Lambda^{j+4}$ such that
\begin{subequations}\label{decoupleAequiv}
\begin{align}
(\dd w,\dd v)+(v,\dd\lambda)&=(f,v), \label{decoupleAequiv1}
 \\
(w,\dd\eta)&=0, \label{decoupleAequiv10} \\
\label{decoupleAequiv2}
 (\nabla\phi,\nabla\psi)+(\dd  \psi, p)
 +(s,\dd p)&=(\dd w,\psi), \\
\label{decoupleAequiv3}
 (\dd \phi, q)+(r, \dd q) &=0,\\
\label{decoupleAequiv4}
(\dd u,\dd \chi)+(\chi,\dd z)&=(\phi,\dd\chi), \\
\label{decoupleAequiv40}
(u,\dd\mu)&=(g,\mu), 
\end{align}
\end{subequations}
for any $v,\chi\in H_{0}\Lambda^{j}$, $\eta,\mu\in H_0\Lambda^{j-1}/\dd H_0\Lambda^{j-2}$,
$\psi\in H_{0}^{1}\Lambda^{j+1}$,
$ q\in H_0\Lambda^{j+2}$ and $s\in L^2\Lambda^{j+3}/\delta H^*\Lambda^{j+4}$.
The decoupled formulation \eqref{decoupleAequiv} is related to complex
\begin{equation*}
H_0\Lambda^{j-1}\xrightarrow{\dd} H_0^{\rm Gd}\Lambda^j\xrightarrow{\dd}H_{0}^{1}\Lambda^{j+1} \xrightarrow{\dd} H_{0}\Lambda^{j+2} \xrightarrow{\dd} L^{2}\Lambda^{j+3}/\delta H^*\Lambda^{j+4} \to 0,
\end{equation*}
rather than complex \eqref{H2derhamcomplex1}.
The numerical methods for the biharmonic equation in~\cite{Gallistl2017} are based on the decoupled formulation \eqref{decoupleAequiv} with $j=0$ and $d=2,3$.
\end{remark}

\subsection{Decoupled formulation without quotient spaces}
Due to the presence of the quotient spaces $H_0\Lambda^{j-1}/\dd H_0\Lambda^{j-2}$ and $L^2\Lambda^{j+3}/\delta H^*\Lambda^{j+4}$ in the decoupled formulation \eqref{decoupleform}, its discretization is less straightforward. To this end, by using the fact that $r=0$ and $\lambda=z=0$ in \eqref{decoupleform}, we propose the following equivalent decoupled formulation without quotient spaces:
find $u, w\in H_{0}\Lambda^{j}$, $\lambda,z\in H_0\Lambda^{j-1}$,
$\phi\in H_{0}^{1}\Lambda^{j+1}$,
$ p\in L^{2}\Lambda^{j+2}$ and $r\in H^*\Lambda^{j+3}$ such that
\begin{subequations}\label{decoupleformnew}
\begin{align}
(\dd w,\dd v)+(v,\dd\lambda)&=(f,v)\qquad\;\forall~v\in H_{0}\Lambda^{j}, \label{decoupledformnew1}
 \\
(w,\dd\eta)-(\lambda,\eta)&=0\qquad\qquad\, \forall~\eta\in H_0\Lambda^{j-1}, \label{decoupledformnew10}
 \\
 (\nabla\phi,\nabla\psi)+(r,s)+(\dd\psi+\delta s, p)&=(\dd w,\psi)\quad\;\forall~\psi\in H_{0}^{1}\Lambda^{j+1}, s\in H^*\Lambda^{j+3}, \label{decoupledformnew2}\\
 (\dd \phi+\delta r, q) &=0\qquad\qquad\,\forall~q\in L^2\Lambda^{j+2},\label{decoupledformnew3}\\
  (\dd u,\dd \chi)+(\chi,\dd z)&=(\phi,\dd\chi)\quad\;\,\forall~\chi\in H_{0}\Lambda^{j}, \label{decoupledformnew4} \\
(u,\dd\mu)-(z,\mu)&=(g,\mu)\qquad\;\forall~\mu\in H_0\Lambda^{j-1}. \label{decoupledformnew40}
\end{align}
\end{subequations}

\begin{lemma}\label{lem:stokesge}
Let $0\leq j\leq d-1$.
The formulation \eqref{decoupledformnew2}-\eqref{decoupledformnew3} of the generalized Stokes equation is well-posed, and $r=0$.
\end{lemma}
\begin{proof}
When $j=d-1$, the formulation \eqref{decoupledformnew2}-\eqref{decoupledformnew3} reduces to the weak formulation of Poisson equation, which is well-posed.

Now consider case $0\leq j\leq d-2$.
For $(\psi, s)\in  H_{0}^{1}\Lambda^{j+1}\times H^*\Lambda^{j+3}$ satisfying $(\dd\psi+\delta s, q)=0$ for all $ q\in L^{2}\Lambda^{j+2}$, by Helmholtz decomposition \eqref{eq:2formHelmholtzdecomp1}, we have $\dd\psi=\delta s=0$ and the coercivity
\begin{equation*}
\|\psi\|_{1}^{2}+\|s\|_{H^*\Lambda^{j+3}}^{2}=\|\psi\|_{1}^{2}+\|s\|^{2}\lesssim |\psi|_{1}^{2}+\|s\|^{2}.
\end{equation*}
By \eqref{AH}, it holds the inf-sup condition
\begin{equation*}
\| q\|\lesssim \sup _{\psi\in  H_{0}^{1}\Lambda^{j+1},\,
s\in H^*\Lambda^{j+3}}
\frac{(\dd\psi+\delta s, q)}{\|\psi\|_{1}+\|s\|_{H^*\Lambda^{j+3}}}\quad\forall~q\in L^{2}\Lambda^{j+2}.
\end{equation*}
By applying the Babu\v{s}ka-Brezzi theory \cite{BoffiBrezziFortin2013}, the well-posedness of formulation \eqref{decoupledformnew2}-\eqref{decoupledformnew3} follows from the last coercivity and inf-sup condition.

Equation \eqref{decoupledformnew3} implies that $\delta r = 0$. By setting $s = r$ and $\psi = 0$ in equation~\eqref{decoupledformnew2}, we immediately deduce that $r = 0$.
\end{proof}

\begin{theorem}\label{thm:unisoldecoupledformnew}
Let $0\leq j\leq d-1$.
The decoupled formulation \eqref{decoupleformnew} is well-posed, and equivalent to both the decoupled formulation \eqref{decoupleform} and the variational formulation \eqref{4thorderpdeweakform}.
In particular, if $(w,\lambda) \in H_{0}\Lambda^{j} \times H_0\Lambda^{j-1}$ solves problem \eqref{decoupledformnew1}-\eqref{decoupledformnew10}, $(\phi,p,r) \in H_{0}^{1}\Lambda^{j+1} \times L^{2}\Lambda^{j+2} \times H^*\Lambda^{j+3}$ solves problem \eqref{decoupledformnew2}-\eqref{decoupledformnew3}, and $(u,z) \in H_{0}\Lambda^{j} \times H_0\Lambda^{j-1}$ solves problem \eqref{decoupledformnew4}-\eqref{decoupledformnew40}, then these solutions also satisfy problems \eqref{decoupledform1}-\eqref{decoupledform10}, \eqref{decoupledform2}-\eqref{decoupledform3}, and \eqref{decoupledform4}-\eqref{decoupledform40}, respectively. Moreover, we have $r = 0$, $\lambda = z = 0$, $p\in \dd H_0^1\Lambda^{j+1}$ and $\phi=\dd u$.
\end{theorem}
\begin{proof}
The well-posedness of problems \eqref{decoupledformnew1}-\eqref{decoupledformnew10} and \eqref{decoupledformnew4}-\eqref{decoupledformnew40}, along with their equivalence to problems \eqref{decoupledform1}-\eqref{decoupledform10} and \eqref{decoupledform4}-\eqref{decoupledform40}, respectively, can be shown following the proof of Lemma~\ref{lem:4thorderpdeweakformag}.  Lemma~\ref{lem:stokesge} establishes the well-posedness of problem \eqref{decoupledformnew2}-\eqref{decoupledformnew3} and demonstrates that $r=0$.  Consequently, problem \eqref{decoupledformnew2}-\eqref{decoupledformnew3} is equivalent to problem \eqref{decoupledform2}-\eqref{decoupledform3}.
\end{proof}

Unlike the decoupled formulation \eqref{decoupleform}, the new decoupled formulation \eqref{decoupleformnew} avoids the use of quotient spaces. This makes it more amenable to designing finite element methods and developing fast solvers compared to both the decoupled formulation \eqref{decoupleform} and the variational formulation \eqref{4thorderpdeweakform}.

\begin{example}\rm
Taking $j=0$, we get the decoupling of the biharmonic equation~\eqref{biharmoniceqn}: find $u, w\in H_{0}^1(\Omega)$,
$\phi\in H_{0}^{1}\Lambda^{1}$,
$ p\in L^{2}\Lambda^{2}$ and $r\in H^*\Lambda^{3}$ such that
\begin{align*}
(\nabla w, \nabla v) &=(f,v) \qquad\quad \forall~v\in H_{0}^1(\Omega),
 \\
 (\nabla\phi,\nabla\psi)+(r,s)+(\curl\psi+\delta s, p)&=(\nabla w,\psi) \quad\;\;\; \forall~\psi\in H_{0}^{1}\Lambda^{1}, s\in H^*\Lambda^{3},\\
 (\curl \phi+\delta r, q) &=0 \qquad\qquad\;\;\, \forall~q\in L^2\Lambda^{2}, \\
  (\nabla u, \nabla \chi)&=(\phi,\nabla\chi) \qquad \forall~\chi\in H_{0}^1(\Omega).
\end{align*}
\end{example}

\begin{example}\rm
Taking $j=1$, we get the decoupling of the quad-curl problem \eqref{quadcurleqn}: find $u, w\in H_{0}(\curl,\Omega)$, $\lambda,z\in H_0^1(\Omega)$,
$\phi\in H_{0}^{1}\Lambda^{2}$,
$p\in L^{2}\Lambda^{3}$ and $r\in H^*\Lambda^{4}$ such that
\begin{align*}
(\curl w, \curl v)+(v, \nabla\lambda)&=(f,v) \qquad\qquad \forall~v\in H_{0}(\curl,\Omega), 
 \\
(w, \nabla\eta)-(\lambda,\eta)&=0 \qquad\qquad\quad\;\;\, \forall~\eta\in H_{0}^1(\Omega), 
 \\
 (\nabla\phi,\nabla\psi)+(r,s)+(\dd \psi+\delta s, p)&=(\curl w,\psi) \quad\;\;\; \forall~\psi\in H_{0}^{1}\Lambda^{2}, s\in H^*\Lambda^{4}, \\
 (\dd \phi+\delta r, q) &=0 \qquad\qquad\quad\;\;\, \forall~q\in L^{2}\Lambda^{3},\\
  (\curl u,\curl \chi)+(\chi,\nabla z)&=(\phi, \curl\chi) \qquad \forall~\chi\in H_{0}(\curl,\Omega),  \\
(u,\nabla\mu)-(z,\mu)&=(g,\mu) \qquad\quad\;\;\;\, \forall~\mu\in H_{0}^1(\Omega).
\end{align*}
When $d=3$, the middle formulation is: find $\phi\in H_{0}^{1}(\Omega;\mathbb R^3)$,
$ p\in L^{2}(\Omega)$ and $r\in \mathbb R$ such that
\begin{align*}
 (\nabla\phi,\nabla\psi)+(\div\psi+s, p)&=(\curl w,\psi) \quad\;\;\; \forall~\psi\in H_{0}^{1}(\Omega;\mathbb R^3), s\in\mathbb R, \\
 (\div\phi+r, q) &=0 \qquad\qquad\quad\;\;\; \forall~q\in L^{2}(\Omega).
\end{align*}
\end{example}

\begin{example}\rm
Taking $j=d-1$, we get the decoupling of the fourth-order div problem \eqref{quaddiveqn}: 
find $u, w\in H_{0}(\div,\Omega)$, $\lambda,z\in H_0(\curl,\Omega)$
and $\phi\in H_{0}^{1}(\Omega)\cap L_0^2(\Omega)$ such that
\begin{align*}
(\div w,\div v)+(v,\curl\lambda)&=(f,v) \qquad\qquad \forall~v\in H_{0}(\div,\Omega), 
 \\
(w,\curl\eta)-(\lambda,\eta)&=0 \qquad\qquad\quad\;\;\, \forall~\eta\in H_0(\curl,\Omega), 
 \\
 (\nabla\phi,\nabla\psi)&=(\div w,\psi) \qquad \forall~\psi\in H_{0}^{1}(\Omega)\cap L_0^2(\Omega), \\
  (\div u,\div \chi)+(\chi,\curl z)&=(\phi, \div\chi) \qquad\, \forall~\chi\in H_{0}(\div,\Omega),  \\
(u,\curl\mu)-(z,\mu)&=(g,\mu) \qquad\quad\;\;\; \forall~\mu\in H_0(\curl,\Omega). 
\end{align*}
\end{example}

\section{Decoupled finite element methods}\label{sec:MINIfem}

We will develop a family of conforming finite element methods for the decoupled formulation \eqref{decoupleformnew} in this section, where the finite element pair for the generalized Stokes equation can be regarded as the generalization of the MINI element for Stokes equation in \cite{ArnoldBrezziFortin1984,BoffiBrezziFortin2013}.

\subsection{Finite element spaces and interpolations}
We first construct conforming finite element spaces for the generalized Stokes equation~\eqref{decoupledformnew2}-\eqref{decoupledformnew3}.
To discretize $H_{0}^{1}\Lambda^{j+1}(\Omega)$ with $0\leq j\leq d-1$, for simplex $T$ and integer $k\geq 1$, we take 
\begin{equation*}
\Phi_k(T) :=\mathbb{P}_{k}\Lambda^{j+1}(T) + b_{T}\,\delta\mathbb{P}_{k}\Lambda^{j+2}(T)
\end{equation*}
as the space of shape functions,
where $b_T=\lambda_0\lambda_1\cdots\lambda_d$ is the bubble function,
and $\lambda_i (i = 0,1, \ldots, d) $ are the barycentric coordinates corresponding to the vertices of $T$.
When $k=1$,  by $\delta\mathbb{P}_{1}\Lambda^{j+2}(T)=\mathbb{P}_{0}\Lambda^{j+1}(T)$, we have 
\begin{equation*}
\Phi_1(T)=\mathbb{P}_{1}\Lambda^{j+1}(T) + b_{T}\,\mathbb{P}_{0}\Lambda^{j+1}(T).
\end{equation*}
When $j=d-1$, $\Phi_k(T) =\mathbb{P}_{k}\Lambda^{d}(T) + b_{T}\mathbb{P}_{0}\Lambda^{d}(T)$.
For $j\leq d-2$, it holds
\begin{equation*}
\Phi_k(T) =\mathbb{P}_{k}\Lambda^{j+1}(T) \oplus b_{T}\,\delta(\mathbb{P}_{k}\Lambda^{j+2}(T)/\mathbb{P}_{k-d}\Lambda^{j+2}(T)).
\end{equation*}

The degrees of freedom (DoFs) are given by
\begin{subequations}\label{MINIDoFs}
\begin{align}
& v(\texttt{v}),  &&\,\texttt{v}\in \Delta_{0}(T), \label{MA}\\
&( v, q)_{f},  &&\,  q
\in \mathbb{P}_{k-\ell-1}\Lambda^{j+1}(f),\,f\in \Delta_{\ell}(T), \ell=1, \ldots, d-1, \label{MB}\\
&( v, q)_{T},  &&\,  q\in \mathbb{P}_{k-d-1}\Lambda^{j+1}(T)+\delta\mathbb{P}_{k}\Lambda^{j+2}(T).\label{MD}
\end{align}
\end{subequations}
Thanks to the decomposition \eqref{eq:polydecompdelta}, DoF \eqref{MD} can be rewritten as
\begin{equation*}
( v, q)_{T},\quad  q\in \star\kappa\mathbb{P}_{k-d-2}\Lambda^{d-j}(T)\oplus\delta \mathbb{P}_{k}\Lambda^{j+2}(T).
\end{equation*}
The DoFs \eqref{MA}-\eqref{MB} and the first part of DoF \eqref{MD} correspond to the DoFs of Lagrange element.

\begin{lemma}\label{lemOA}
Let $0\leq j\leq d-1$.
The DoFs \eqref{MINIDoFs} are unisolvent for space $\Phi_k(T)$.
\end{lemma}
\begin{proof}
Both the number of DoFs \eqref{MINIDoFs} and $\dim \Phi_k(T)$ are $1+\dim\mathbb{P}_{k}(T)$ for the case $j=d-1$ and $d=k$. In all other instances, the dimensions are described by
$$\dim \mathbb{P}_{k}\Lambda^{j+1}(T)+\dim \delta(\mathbb{P}_{k}\Lambda^{j+2}(T)/\mathbb{P}_{k-d}\Lambda^{j+2}(T)).$$

Take $ v \in \Phi_k(T)$ satisfying all the DoFs \eqref{MINIDoFs} vanish. 
Since $ v|_F \in \mathbb{P}_{k}\Lambda^{j+1}(F)$ for $F\in\Delta_{d-1}(T)$, 
the vanishing DoFs (\ref{MA})-(\ref{MB}) imply $ v|_{\partial T }= 0$. 
Then $ v = b_T q$ with $ q\in \mathbb{P}_{k-d-1}\Lambda^{j+1}(T)+\delta \mathbb{P}_{k}\Lambda^{j+2}(T)$, which together with the vanishing DoF (\ref{MD}) indicates $ v  = 0$.
\end{proof}

Let $\{\tau_{\ell}:\ell=1,\ldots,{d\choose j+1}\}$ be a basis of ${\rm Alt}^{j+1}\mathbb R^d$.
We present the basis functions of $\Phi_k(T)$ corresponding to DoFs \eqref{MINIDoFs}:
\begin{enumerate}[label=(\roman*)]
\item A set of basis functions related to DoFs \eqref{MA}-\eqref{MB} is
\begin{equation*}
\varphi_i^L\tau_{\ell}\quad\textrm{ for } i=1,\ldots,N_L^{\partial},\;\ell=1,\ldots, {d\choose j+1},
\end{equation*}
where $\{\varphi_i^L:i=1,\ldots,N_L^{\partial}\}$ is a basis of Lagrange element corresponding to the boundary DoFs \cite{Nicolaides1972,ChenChenHuangWei2024}. 
\item A set of basis functions related to DoF \eqref{MD} is
\begin{equation*}
b_T\varphi_{i}\quad\textrm{ for } i=1,\ldots,N_T,
\end{equation*}
where $\{\varphi_i:i=1,\ldots,N_T\}$ is a basis of $\star\kappa\mathbb{P}_{k-d-2}\Lambda^{d-j}(T)\oplus\delta \mathbb{P}_{k}\Lambda^{j+2}(T)$.
\end{enumerate}

Define the $H^1$-conforming finite element space
$$
\Phi_h:=\{ v_{h}\in H_{0}^{1}\Lambda^{j+1}(\Omega)
 : v_h|_{T}\in\mathbb{P}_{k}\Lambda^{j+1}(T) + b_{T}\,\delta\mathbb{P}_{k}\Lambda^{j+2}(T) \textrm{ for } T\in \mathcal{T}_{h}\}.$$
When $k=1$, $j=d-2$ and $d=2,3$, $\Phi_h$ is exactly the MINI element space in \cite{ArnoldBrezziFortin1984,BoffiBrezziFortin2013} for Stokes equation.
For $j=0,1,\ldots, d$, let 
\begin{align*}
V_{k,h}^{\dd}\Lambda^{j}&:=\mathbb P_{k}\Lambda^{j}(\mathcal T_h)\cap H\Lambda^{j}(\Omega),\quad\;\, V_{k,h}^{\delta}\Lambda^{j}:=\star V_{k,h}^{\dd}\Lambda^{d-j},\\
V_{k,h}^{\dd,-}\Lambda^{j}&:=\mathbb P_{k}^-\Lambda^{j}(\mathcal T_h)\cap H\Lambda^{j}(\Omega),\quad V_{k,h}^{\delta,-}\Lambda^{j}:=\star V_{k,h}^{\dd,-}\Lambda^{d-j}.
\end{align*}
Set $\mathring{V}_{k,h}^{\dd}\Lambda^{j}:=V_{k,h}^{\dd}\Lambda^{j}\cap H_{0}\Lambda^{j}(\Omega)$ and $\mathring{V}_{k,h}^{\dd,-}\Lambda^{j}:=V_{k,h}^{\dd,-}\Lambda^{j}\cap H_{0}\Lambda^{j}(\Omega)$.
We have the short exact sequences \cite{Hiptmair2001,Hiptmair2002,Arnold2018,ArnoldFalkWinther2006}
\begin{equation}\label{femderhamcomplex1form1}
\mathring{V}_{k+1,h}^{\dd,-}\Lambda^{j-1}\xrightarrow{\dd} 
\mathring{V}_{k,h}^{\dd}\Lambda^{j}\xrightarrow{\dd} \dd \mathring{V}_{k,h}^{\dd}\Lambda^{j} \to0,
\end{equation}
\begin{equation}\label{femderhamcomplex1form2}
\mathring{V}_{k,h}^{\dd,-}\Lambda^{j-1}\xrightarrow{\dd} 
\mathring{V}_{k,h}^{\dd,-}\Lambda^{j}\xrightarrow{\dd} \dd \mathring{V}_{k,h}^{\dd,-}\Lambda^{j} \to0,
\end{equation}
\begin{equation}\label{femderhamcomplex1form}
V_{k,h}^{\delta,-}\Lambda^{j+1} \xrightarrow{\delta} 
V_{k,h}^{\delta,-}\Lambda^{j}\xrightarrow{\delta} \delta V_{k,h}^{\delta,-}\Lambda^{j} \to0.
\end{equation}

Let $I_{k,h}^{\dd}: L^2\Lambda^j(\Omega)\to V_{k,h}^{\dd}\Lambda^{j}$ and $I_{k,h}^{\dd,-}: L^2\Lambda^j(\Omega)\to V_{k,h}^{\dd,-}\Lambda^{j}$ be the $L^2$ bounded projection operators devised in~\cite{ArnoldGuzman2021,ChristiansenWinther2008}, then $I_{k,h}^{\dd}\omega, I_{k,h}^{\dd,-}\omega\in H_0\Lambda^j(\Omega)$ for $\omega\in H_0\Lambda^j(\Omega)$.
Define $I_{k,h}^{\delta}: L^2\Lambda^j(\Omega)\to V_{k,h}^{\delta}\Lambda^{j}$ and $I_{k,h}^{\delta,-}: L^2\Lambda^j(\Omega)\to V_{k,h}^{\delta,-}\Lambda^{j}$ as 
\begin{equation*}
\star I_{k,h}^{\delta}\omega=I_{k,h}^{\dd}(\star\omega),\;\; \star I_{k,h}^{\delta,-}\omega=I_{k,h}^{\dd,-}(\star\omega), \quad\omega\in L^2\Lambda^j(\Omega).
\end{equation*}
Then $I_{k,h}^{\delta}$ and $I_{k,h}^{\delta,-}$ are also $L^2$ bounded. 
We will abbreviate $I_{k,h}^{\dd}$ and $I_{k,h}^{\dd,-}$ as $I_{h}^{\dd}$, and $I_{k,h}^{\delta}$ and $I_{k,h}^{\delta,-}$ as $I_{h}^{\delta}$, if not causing any confusions.
We have (cf. \cite{ArnoldGuzman2021,ChristiansenWinther2008})
\begin{equation}\label{eq:IhdCDprop}
\dd(I_h^{\dd}\omega)=I_h^{\dd}(\dd \omega),\quad \omega\in H\Lambda^j(\Omega),\, 0\leq j\leq d,
\end{equation}
\begin{equation}\label{eq:IhdeltaCDprop}
\delta(I_h^{\delta}\omega)=I_h^{\delta}(\delta \omega),\quad \omega\in H^*\Lambda^j(\Omega),\, 0\leq j\leq d,
\end{equation}
\begin{equation}\label{eq:Ihddelta-estimate}
\|\omega-I_{k,h}^{\dd,-}\omega\| + \|\omega-I_{k,h}^{\delta,-}\omega\| \lesssim h^k|\omega|_k,\quad \omega\in H^k\Lambda^j(\Omega),
\end{equation}
\begin{equation}\label{eq:Ihddeltaestimate}
\|\omega-I_{k,h}^{\dd}\omega\| + h\|\dd(\omega-I_{k,h}^{\dd}\omega)\| \lesssim h^{k+1}|\omega|_{k+1},\quad \omega\in H^{k+1}\Lambda^j(\Omega).
\end{equation}

We next introduce a Fortin operator $I_h: H_{0}^{1}\Lambda^{j+1}(\Omega)\rightarrow \Phi_h$ based on DoFs \eqref{MINIDoFs} as follows:
\begin{align*}
I_h v(\texttt{v})
&=\frac{1}{|\mathcal T_{\texttt{v}}|}\sum_{T\in\mathcal T_{\texttt{v}}}(Q_{k,T} v)(\texttt{v}), && \texttt{v}\in\Delta_{0}(\mathring{\mathcal T}_h),
\\
(I_h v, q)_{f}
&=\frac{1}{|\mathcal T_{f}|}\sum_{T\in\mathcal T_{f}}(Q_{k,T} v, q)_{f},
&&  q\in \mathbb{P}_{k-\ell-1}\Lambda^{j+1}(f),\,f\in \Delta_{\ell}(\mathring{\mathcal T}_h), 1\leq\ell\leq d-2,\\
(I_h v, q)_{F}
&=( v, q)_{F},
&&  q\in \mathbb{P}_{k-d}\Lambda^{j+1}(F), F\in\Delta_{d-1}(\mathring{\mathcal T}_h),\\
(I_h v, q)_{T}&=( v, q)_{T}, &&  q\in \mathbb{P}_{k-d-1}\Lambda^{j+1}(T)+\delta\mathbb{P}_{k}\Lambda^{j+2}(T), T\in\mathcal T_h.
\end{align*}

\begin{remark}\rm
Let $j=d-1$.
Since DoF \eqref{MD} includes the moment $\int_Tv\dx$,
we have $I_hv\in L_{0}^{2}\Lambda^{d}(\Omega)$ for $v\in H_{0}^{1}\Lambda^{d}(\Omega)\cap L_{0}^{2}\Lambda^{d}(\Omega)$.
\end{remark}

Applying the integration by parts, we have from the definition of $I_h$ that 
\begin{equation}\label{eq:MINIcurlcd}
(\dd( v-I_h v),  q)=(v-I_h v, \delta q)=0\quad \forall~ v\in H_{0}^{1}\Lambda^{j+1}(\Omega),\,  q\in V_{k,h}^{\delta,-}\Lambda^{j+2}.
\end{equation}

\begin{lemma}
Let $0\leq j\leq d-1$.
Let $ v\in H_{0}^{1}\Lambda^{j+1}(\Omega)\cap H^{s}\Lambda^{j+1}(\Omega)$ with $1\leq s\leq k+1$. We have for $0\leq m\leq s$ and $T\in\mathcal T_h$ that 
\begin{align}\label{BEqs}
| v-I_h v|_{m,T}\lesssim h_{T}^{s-m}| v|_{s,\omega_{T}}.
\end{align}
\end{lemma}
\begin{proof}
Following the argument in \cite{Wang2001,Brenner2003,HuangHuang2011}, we use the inverse inequality, the norm equivalence on the shape function space and the trace inequality to get
\begin{align*}
h_T^m|Q_{k,T} v-I_h v|_{m,T}&\lesssim \|Q_{k,T} v-I_h v\|_{T} \\
&\lesssim h_{T}^{1/2}\|Q_{k,T} v- v\|_{\partial T}+ \sum_{\texttt{v}\in\Delta_0(T)}h_{T}^{d/2}|(Q_{k,T} v-I_h v)(\texttt{v})| \\
&\quad+ \sum_{\ell=1}^{d-2}\sum_{f\in\Delta_{\ell}(T)}h_{T}^{(d-\ell)/2}\|Q_{k-\ell-1,f}(Q_{k,T} v-I_h v)\|_{F} \\
&\lesssim \sum_{\texttt{v}\in\Delta_0(T)}\sum_{K\in\mathcal T_{\texttt{v}}}(\|Q_{k,K} v- v\|_{K}+h_{T}|Q_{k,K} v- v|_{1,K}),
\end{align*}
which together with the triangle inequality and the error estimate of $Q_{k,K}$ to derive~\eqref{BEqs}.
\end{proof}

\subsection{Decoupled discrete methods}
Now we propose a family of finite element methods for the decoupled formulation \eqref{decoupleformnew} with $0\leq j\leq d-1$: find $w_{h}\in \mathring{V}_{k,h}^{\dd}\Lambda^{j}$,
$\phi_{h}\in\Phi_h$,
$p_h\in V_{k,h}^{\delta,-}\Lambda^{j+2}$, $r_h\in V_{k,h}^{\delta,-}\Lambda^{j+3}$,
$u_h\in \mathring{V}_{k+1,h}^{\dd,-}\Lambda^{j}$ and $\lambda_h,z_h\in \mathring{V}_{k+1,h}^{\dd,-}\Lambda^{j-1}$ such that
\begin{subequations}\label{decoupleMINI}
\begin{align}
(\dd w_h,\dd v)+(v,\dd\lambda_h)&=(f,v)
  &&\forall \,\, v\in \mathring{V}_{k,h}^{\dd}\Lambda^{j},\label{decoupleMINI1}\\
(w_h,\dd\eta)-(\lambda_h,\eta)&=0 &&\forall \,\, \eta\in \mathring{V}_{k+1,h}^{\dd,-}\Lambda^{j-1}, \label{decoupleMINI10}
 \\
(\nabla\phi_{h},\nabla\psi)+(r_h,s)+(\dd \psi
 +\delta s, p_{h})&=(\dd w_{h},\psi)
  &&\forall \,\, \psi\in \Phi_h, s\in V_{k,h}^{\delta,-}\Lambda^{j+3},\label{decoupleMINI2} \\
 (\dd \phi_{h}+\delta r_{h}, q) &=0  &&\forall \,\, q\in V_{k,h}^{\delta,-}\Lambda^{j+2},\label{decoupleMINI3}\\
  (\dd u_h,\dd \chi)+(\chi,\dd z_h)&=(\phi_{h},\dd\chi) &&\forall \,\, \chi\in \mathring{V}_{k+1,h}^{\dd,-}\Lambda^{j},\label{decoupleMINI4} \\
(u_h,\dd\mu)-(z_h,\mu)&=(g,\mu) &&\forall \,\,  \mu\in \mathring{V}_{k+1,h}^{\dd,-}\Lambda^{j-1}. \label{decoupleMINI40}
\end{align}
Here $V_{k,h}^{\delta,-}\Lambda^{j+2}=\{0\}$ and $V_{k,h}^{\delta,-}\Lambda^{j+3}=\{0\}$ for $j=d-1$.
\end{subequations}

For simplicity, introduce the bilinear forms 
$$
a(\phi,r;\psi,s)=(\nabla\phi,\nabla\psi)+(r,s),\quad
b(\psi, s; q)=(\dd\psi+\delta s, q).
$$
We present the discrete coercivity in the following lemma.
\begin{lemma}
Let $0\leq j\leq d-2$.
For $(\psi, s)\in \Phi_h \times V_{k,h}^{\delta,-}\Lambda^{j+3}$ satisfying $b(\psi, s; q)=0$ for all $q\in V_{k,h}^{\delta,-}\Lambda^{j+2}$, we have the discrete coercivity
\begin{align}\label{BE2}
\|\psi\|_{1}^{2}+\|s\|_{H^*\Lambda^{j+3}}^{2}\lesssim a(\psi,s;\psi,s).
\end{align}
\end{lemma}
\begin{proof}
By complex \eqref{femderhamcomplex1form}, take $ q=\delta s$ in $b(\psi, s; q)=0$ to acquire $\delta s=0$.
Then \eqref{BE2} holds from the Poincar\'e inequality.
\end{proof}

To prove the discrete inf-sup condition, we first establish an $L^2$-stable decomposition of space $V_{k,h}^{\delta,-}\Lambda^{j+2}$ with the help of $I_h^{\delta}$.
\begin{lemma}\label{lem:L2stabledecompNedelec}
Let $0\leq j\leq d-2$.
For $ p_h\in V_{k,h}^{\delta,-}\Lambda^{j+2}$, these exist $ q_h\in V_{k,h}^{\delta,-}\Lambda^{j+2}$ and $r_h\in V_{k,h}^{\delta,-}\Lambda^{j+3}$ such that 
\begin{equation*}
 p_h= q_h+\delta r_h,
\end{equation*}
\begin{equation*}
\| p_h\|\eqsim \| q_h\|+\|\delta r_h\|, \quad \| q_h\|\lesssim \|\delta p_h\|_{-1}.
\end{equation*}
\end{lemma}
\begin{proof}
By Theorem 1.1 in \cite{CostabelMcIntosh2010}, there exists a $ q\in H^*\Lambda^{j+2}(\Omega)$ such that 
\begin{equation*}
\delta q=\delta p_h,\quad \| q\|\lesssim \|\delta p_h\|_{-1}.
\end{equation*}
Set $ q_h=I_h^{\delta} q\in V_{k,h}^{\delta,-}\Lambda^{j+2}$. Since $I_h^{\delta}$ is an $L^2$ bounded projection operator and $\delta q\in\delta V_{k,h}^{\delta,-}\Lambda^{j+2}$, by \eqref{eq:IhdeltaCDprop} we have
\begin{equation*}
\delta q_h=I_h^{\delta}(\delta q)=\delta p_h,\quad \| q_h\|\lesssim \| q\|\lesssim \|\delta p_h\|_{-1}\lesssim \| p_h\|.
\end{equation*}
By complex \eqref{femderhamcomplex1form}, we obtain $ p_h= q_h+\delta r_h$ with $r_h\in V_{k,h}^{\delta,-}\Lambda^{j+3}/\delta V_{k,h}^{\delta,-}\Lambda^{j+4}$, which ends the proof.
\end{proof}

\begin{lemma}
Let $0\leq j\leq d-2$.
It holds the $L^2$ norm equivalence
\begin{equation}\label{eq:L2normequivNedelec}
\| p_h\|\eqsim \|\delta p_h\|_{-1}+\sup_{s\in V_{k,h}^{\delta,-}\Lambda^{j+3}}\frac{( p_h, \delta s)}{\|s\|_{H^*\Lambda^{j+3}}}\quad\;\;\forall~ p_h\in V_{k,h}^{\delta,-}\Lambda^{j+2}.
\end{equation}
\end{lemma}
\begin{proof}
Apply Lemma~\ref{lem:L2stabledecompNedelec} to $ p_h$, then 
\begin{equation*}
\| p_h\|\leq \| q_h\|+\|\delta r_h\|\lesssim \|\delta p_h\|_{-1}+\|\delta r_h\|.
\end{equation*}
On the other side, by the discrete Poincar\'e inequality \cite{Arnold2018,ArnoldFalkWinther2010,ArnoldFalkWinther2006},
\begin{align*}
\|\delta r_h\|&=\sup_{s\in V_{k,h}^{\delta,-}\Lambda^{j+3}/\delta V_{k,h}^{\delta,-}\Lambda^{j+4}}\frac{(\delta r_h, \delta s)}{\|\delta s\|}\lesssim \sup_{s\in V_{k,h}^{\delta,-}\Lambda^{j+3}/\delta V_{k,h}^{\delta,-}\Lambda^{j+4}}\frac{(\delta r_h, \delta s)}{\|s\|_{H^*\Lambda^{j+3}}} \\
&\leq \sup_{s\in V_{k,h}^{\delta,-}\Lambda^{j+3}}\frac{( p_h- q_h, \delta s)}{\|s\|_{H^*\Lambda^{j+3}}} \leq \| q_h\| + \sup_{s\in V_{k,h}^{\delta,-}\Lambda^{j+3}}\frac{( p_h, \delta s)}{\|s\|_{H^*\Lambda^{j+3}}} \\
&\lesssim \|\delta  p_h\|_{-1} + \sup_{s\in V_{k,h}^{\delta,-}\Lambda^{j+3}}\frac{( p_h, \delta s)}{\|s\|_{H^*\Lambda^{j+3}}}.
\end{align*}
Combining the last two inequalities yields \eqref{eq:L2normequivNedelec}.
\end{proof}


We are now in a position to derive the discrete inf-sup condition.
\begin{lemma}
Let $0\leq j\leq d-2$.
For $ q\in V_{k,h}^{\delta,-}\Lambda^{j+2}$, it holds the discrete inf-sup condition
\begin{align}\label{BE1}
\| q\|\lesssim \sup _{\psi\in \Phi_h,\,
s\in V_{k,h}^{\delta,-}\Lambda^{j+3}}
\frac{b(\psi, s; q)}{\|\psi\|_{1}+\|s\|_{H^*\Lambda^{j+3}}}.
\end{align}
\end{lemma}
\begin{proof}
Employ \eqref{eq:MINIcurlcd} and \eqref{BEqs} with $m=s=1$ to have
\begin{align*}
\|\delta  q\|_{-1}
&=\sup_{\phi\,\in  H_{0}^{1}\Lambda^{j+1}(\Omega)}
\frac{(\dd\phi,  q)}{\|\phi\|_{1}}
\,=\sup_{\phi\,\in  H_{0}^{1}\Lambda^{j+1}(\Omega)}
\frac{(\dd (I_{h}\phi),  q)}{\|\phi\|_{1}} \\
&\lesssim \sup _{\phi\,\in  H_{0}^{1}\Lambda^{j+1}(\Omega)}
\frac{(\dd (I_{h}\phi),  q)}{\|I_{h}\phi\|_{1}}
\,\leq \sup _{\psi\,\in \Phi_h}
\frac{(\dd \psi,  q)}{\|\psi\|_{1}} \\
&\leq \sup _{\psi\,\in \Phi_h,\,\,
s\,\in V_{k,h}^{\delta,-}\Lambda^{j+3}}
\frac{b(\psi, s; q)}{\|\psi\|_{1}+\|s\|_{H^*\Lambda^{j+3}}}.
\end{align*}
On the other hand,
\begin{align*}
\sup_{s\in V_{k,h}^{\delta,-}\Lambda^{j+3}}\frac{(q, \delta s)}{\|s\|_{H^*\Lambda^{j+3}}}
\,\leq \sup _{\psi\,\in \Phi_h,\,\,
s\,\in V_{k,h}^{\delta,-}\Lambda^{j+3}}
\frac{b(\psi, s; q)}{\|\psi\|_{1}+\|s\|_{H^*\Lambda^{j+3}}}.
\end{align*}
Therefore, \eqref{BE1} follows from \eqref{eq:L2normequivNedelec} and the last two inequalities.
\end{proof}

\begin{theorem}\label{thm:dfemMINIunisol}
For $0\leq j\leq d-1$, 
the decoupled finite element method \eqref{decoupleMINI} is well-posed. Moreover, $r_{h}=0$ and $\lambda_{h}=z_{h}=0$.
\end{theorem}
\begin{proof}
Both \eqref{decoupleMINI1}-\eqref{decoupleMINI10} and \eqref{decoupleMINI4}-\eqref{decoupleMINI40} are the mixed finite element methods for second-order exterior differential equations, whose well-posedness follows from the exactness of complexes \eqref{femderhamcomplex1form1}-\eqref{femderhamcomplex1form2} and the proof of Lemma~\ref{lem:4thorderpdeweakformag}.

When $j=d-1$, 
the discretizaiton \eqref{decoupleMINI2}-\eqref{decoupleMINI3} is a conforming finite element method of the Poisson equation, which is well-posed.
For $0\leq j\leq d-2$, 
by applying the Babu\v{s}a-Brezzi theory \cite{BoffiBrezziFortin2013}, the well-posedness of the mixed finite element method~\eqref{decoupleMINI2}-\eqref{decoupleMINI3} follows from the the discrete coercivity (\ref{BE2}),
and the discrete inf-sup condition (\ref{BE1}).

Choose $q=\delta r_h$ in equation \eqref{decoupleMINI3} to get $\delta r_h = 0$. By setting $s = r_h$ and $\psi = 0$ in equation~\eqref{decoupleMINI2}, we acquire that $r_h = 0$.
Take $v=\dd\lambda_h$ in equation \eqref{decoupleMINI1} to have $\dd\lambda_h = 0$.
Then we derive $\lambda_h = 0$ by choosing $\eta=\lambda_h$ in equation \eqref{decoupleMINI10}.
Apply the similar argument to achieve $z_h=0$ from the mixed method \eqref{decoupleMINI4}-\eqref{decoupleMINI40}.
\end{proof}

\begin{example}\rm
Taking $j=0$ and $d=3$, the mixed finite element method \eqref{decoupleMINI2}-\eqref{decoupleMINI3} is to find
$\phi_{h}\in\Phi_h$, $p_h\in V_{k,h}^{\rm ND}$ and $r_h\in V_{k,h}^{\grad}/\mathbb R$ such that
\begin{align*}
 (\nabla\phi_h,\nabla\psi)+(\curl\psi+\nabla s, p_h)&=(\nabla w_h,\psi) \quad\;\; \forall~\psi\in \Phi_h, s\in V_{k,h}^{\grad}/\mathbb R,  \\
 (\curl \phi_h+\nabla r_h, q) &=0  \qquad\qquad\quad \forall~ q\in V_{k,h}^{\rm ND},
\end{align*}
where $V_{k,h}^{\grad}=V_{k,h}^{\delta,-}\Lambda^{3}$ is the Lagrange element space, and $V_{k,h}^{\rm ND}=V_{k,h}^{\delta,-}\Lambda^{2}$ is the N\'{e}d\'{e}lec element space of the first kind \cite{Nedelec1980}.
\end{example}

\begin{example}\rm
Taking $j=1$ and $d=3$, the mixed finite element method \eqref{decoupleMINI2}-\eqref{decoupleMINI3} is to find
$\phi_{h}\in\Phi_h$ and $p_h\in V_{k,h}^{\grad}/\mathbb R$ such that
\begin{align*}
 (\nabla\phi_h,\nabla\psi)+(\div\psi, p_h)&=(\curl w_h,\psi) \quad\;\; \forall~\psi\in \Phi_h,  \\
 (\div\phi_h, q) &=0  \qquad\qquad\qquad \forall~ q\in V_{k,h}^{\grad}/\mathbb R.
\end{align*}
This is exactly the MINI element method for Stokes equation in \cite[Section 8.7.1]{BoffiBrezziFortin2013}.
\end{example}

\begin{example}\rm
Taking $j=d-1$, the finite element method \eqref{decoupleMINI2}-\eqref{decoupleMINI3} is to find
$\phi_{h}\in\Phi_h/\mathbb R$ such that
\begin{equation*}
(\nabla\phi_{h},\nabla\psi) =(\div w_{h},\psi)
  \quad\forall~\psi\in \Phi_h/\mathbb R.
\end{equation*}
This is a conforming finite element method of the Poisson equation.
\end{example}



Furthermore, we consider the local elimination of $\lambda_h$, $r_h$ and $z_h$ in the decoupled finite element method \eqref{decoupleMINI}.
For finite element space $V_h=\mathring{V}_{k+1,h}^{\dd,-}\Lambda^{j-1}$ or $V_{k,h}^{\delta,-}\Lambda^{j+3}$, let $\{\varphi_i^V\}_1^{N_V}$ be the basis functions of $V_h$, and introduce the discrete inner product
\begin{equation*}
\langle \eta, \mu\rangle_D:=\sum_{i=1}^{N_V}\eta_i\mu_i\|\varphi_i^V\|^2, \quad\textrm{ where } \eta=\sum_{i=1}^{N_V}\eta_i\varphi_i^V,\, \mu=\sum_{i=1}^{N_V}\mu_i\varphi_i^V.
\end{equation*}
This inner product induces the following norm equivalence:
\begin{equation}\label{eq:discreteL2normequiv}
\|\mu\|^2\eqsim \langle \mu, \mu\rangle_D\quad\forall~\mu\in V_h.
\end{equation}
The matrix representation of $\langle \eta, \mu\rangle_D$ is simply the diagonal matrix of the mass matrix of $(\eta, \mu)$.

Thanks to $r_{h}=0$ and $\lambda_{h}=z_{h}=0$, the decoupled finite element method \eqref{decoupleMINI} is equivalent to finding $w_{h}\in \mathring{V}_{k,h}^{\dd}\Lambda^{j}$,
$\phi_{h}\in\Phi_h$,
$p_h\in V_{k,h}^{\delta,-}\Lambda^{j+2}$, $r_h\in V_{k,h}^{\delta,-}\Lambda^{j+3}$,
$u_h\in \mathring{V}_{k+1,h}^{\dd,-}\Lambda^{j}$ and $\lambda_h,z_h\in \mathring{V}_{k+1,h}^{\dd,-}\Lambda^{j-1}$ such that
\begin{subequations}\label{decoupleMINIdiag}
\begin{align}
(\dd w_h,\dd v)+(v,\dd\lambda_h)&=(f,v)
  &&\!\!\forall \, v\in \mathring{V}_{k,h}^{\dd}\Lambda^{j},\label{decoupleMINIdiag1}\\
(w_h,\dd\eta)-\langle\lambda_h,\eta\rangle_D&=0 &&\!\!\forall \, \eta\in \mathring{V}_{k+1,h}^{\dd,-}\Lambda^{j-1}, \label{decoupleMINIdiag10} \\
(\nabla\phi_{h},\nabla\psi)+\langle r_h,s\rangle_D+(\dd \psi
 +\delta s, p_{h})&=(\dd w_{h},\psi)
  &&\!\!\forall \, \psi\in \Phi_h, s\in V_{k,h}^{\delta,-}\Lambda^{j+3},\label{decoupleMINIdiag2} \\
 (\dd \phi_{h}+\delta r_{h}, q) &=0  &&\!\!\forall \, q\in V_{k,h}^{\delta,-}\Lambda^{j+2},\label{decoupleMINIdiag3}\\
  (\dd u_h,\dd \chi)+(\chi,\dd z_h)&=(\phi_{h},\dd\chi) &&\!\!\forall \, \chi\in \mathring{V}_{k+1,h}^{\dd,-}\Lambda^{j},\label{decoupleMINIdiag4} \\
(u_h,\dd\mu)-\langle z_h,\mu\rangle_D&=(g,\mu) &&\!\!\forall \,  \mu\in \mathring{V}_{k+1,h}^{\dd,-}\Lambda^{j-1}. \label{decoupleMINIdiag40}
\end{align}
\end{subequations}

\begin{theorem}\label{thm:dfemMINIdiagunisol}
For $0\leq j\leq d-1$, 
the decoupled finite element method \eqref{decoupleMINIdiag} is well-posed, and its solution coincides with that of the decoupled finite element method \eqref{decoupleMINI}.
\end{theorem}
\begin{proof}
By applying the similar argument of Theorem~\ref{thm:dfemMINIunisol}, we conclude the result from the discrete $L^2$ norm equivalence \eqref{eq:discreteL2normequiv}.
\end{proof}

Since the matrices of $\langle\lambda_h,\eta\rangle_D$, $\langle r_h,s\rangle_D$ and $\langle z_h,\mu\rangle_D$ are diagonal, we can locally eliminate $\lambda_h$, $r_h$ and $z_h$ when assembling the linear system for the decoupled mixed method \eqref{decoupleMINIdiag}. The remaining unknowns are only $w_h$, $\phi_h$, $p_h$ and $u_h$. After this elimination, the coefficient matrices of the mixed method \eqref{decoupleMINIdiag1}-\eqref{decoupleMINIdiag10} and the mixed method \eqref{decoupleMINIdiag4}-\eqref{decoupleMINIdiag40} become symmetric and positive definite. Similar techniques have been applied for the quad-curl problem in \cite{Huang2023}, and for the linear algebraic system of the Maxwell equation in \cite[(77)-(78)]{ChenWuZhongZhou2018}.

\subsection{Error analysis}
Next we present the error analysis of the decoupled finite element method \eqref{decoupleMINI}.
\begin{lemma}
Let $0\leq j\leq d-1$.
Let $(w, 0)\in H_{0}\Lambda^{j}\times H_{0}\Lambda^{j-1}$ be the solution of problem \eqref{decoupledformnew1}-\eqref{decoupledformnew10}, and $(w_h, 0)\in \mathring{V}_{k,h}^{\dd}\Lambda^{j}\times \mathring{V}_{k+1,h}^{\dd,-}\Lambda^{j-1}$ be the solution of the mixed finite element method \eqref{decoupleMINI1}-\eqref{decoupleMINI10}. Assume $w\in H^{k+1}\Lambda^{j}$. We have
\begin{align}
\label{eq:whphestimate}
\|\dd(w-w_{h})\|&\lesssim h^{k}|\dd w|_{k}, \\
\label{eq:whphestimate0}
\|w-w_{h}\|&\lesssim h^{k}(h|w|_{k+1}+|\dd w|_{k}).
\end{align}
\end{lemma}
\begin{proof}
Subtract \eqref{decoupleMINI1}-\eqref{decoupleMINI10} from \eqref{decoupledformnew1}-\eqref{decoupledformnew10} to get the error equations
\begin{align}
\label{erroreqnwh1}
 (\dd(w-w_{h}),\dd v) &=0 \qquad
  \forall \,\, v\in\mathring{V}_{k,h}^{\dd}\Lambda^{j}, \\
\label{erroreqnwh2}
 (w-w_{h}, \dd\eta)&=0 \qquad \forall \,\, \eta\in \mathring{V}_{k+1,h}^{\dd,-}\Lambda^{j-1}.
\end{align}
Choosing $v=I_h^{\dd}w-w_{h}$ in \eqref{erroreqnwh1}, we acquire from \eqref{eq:IhdCDprop} that
\begin{align*}
\|\dd(w-w_{h})\|^2&=(\dd(w-w_{h}),\dd(w-I_h^{\dd}w)) \lesssim \|\dd(w-w_{h})\|\|\dd w-I_h^{\dd}(\dd w)\|.
\end{align*}
Hence, it follows
\begin{equation*}
\|\dd(w-w_{h})\|\lesssim \|\dd w-I_h^{\dd}(\dd w)\|,
\end{equation*}
which together with \eqref{eq:Ihddeltaestimate} implies \eqref{eq:whphestimate}.

Finally, by employing the triangle inequality, the discrete Poincar\'e inequalities,~\eqref{eq:IhdCDprop} and the interpolation estimates \eqref{eq:Ihddelta-estimate}-\eqref{eq:Ihddeltaestimate}, we conclude estimate \eqref{eq:whphestimate0} from estimate \eqref{eq:whphestimate}.
\end{proof}

\begin{lemma}
Let $0\leq j\leq d-1$.
Let $(\phi,  p, 0)\in H_{0}^{1}\Lambda^{j+1}\times L^{2}\Lambda^{j+2}\times H^*\Lambda^{j+3}$ be the solution of problem \eqref{decoupledformnew2}-\eqref{decoupledformnew3}, and $(\phi_h,  p_h, 0)\in \Phi_h\times V_{k,h}^{\delta,-}\Lambda^{j+2}\times V_{k,h}^{\delta,-}\Lambda^{j+3}$ the solution of the mixed finite element method \eqref{decoupleMINI2}-\eqref{decoupleMINI3}. Assume $\phi\in H^{k+1}\Lambda^{j+1}$ and $ p\in H^{k}\Lambda^{j+2}$. We have
\begin{equation}\label{eq:phihphestimate}
\|\phi-\phi_{h}\|_{1}
+\| p- p_h\|\lesssim h^{k}(|\phi|_{k+1}+| p|_{k}) + \|\dd(w- w_{h})\|.
\end{equation}
\end{lemma}
\begin{proof}
Subtract \eqref{decoupleMINI2}-\eqref{decoupleMINI3} from \eqref{decoupledformnew2}-\eqref{decoupledformnew3} to get the error equations
\begin{align}
 (\nabla(\phi-\phi_{h}),\nabla\psi)+(\dd \psi, p- p_{h}) &=(\dd(w- w_{h}),\psi)
  &&\forall \,\, \psi\in \Phi_h,\label{erroreqn1} \\
 (\dd(\phi-\phi_{h}), q)=(\dd(I_h\phi-\phi_{h}), q) &=0  &&\forall \,\, q\in V_{k,h}^{\delta,-}\Lambda^{j+2},\label{erroreqn2} \\
( p- p_{h}, \delta s)&=0
  &&\forall \,\, s\in V_{k,h}^{\delta,-}\Lambda^{j+3}.\label{erroreqn3} 
\end{align}
We have used \eqref{eq:MINIcurlcd} in deriving \eqref{erroreqn2}.
Choosing $\psi=I_h\phi-\phi_{h}$ in \eqref{erroreqn1}, it follows from~\eqref{erroreqn2} that
\begin{equation*}
  (\nabla(\phi-\phi_{h}),\nabla(I_h\phi-\phi_{h}))+(\dd(I_h\phi-\phi_{h}), p-I_h^{\delta} p) =(\dd(w- w_{h}),I_h\phi-\phi_{h}),
\end{equation*}
which implies
\begin{equation}\label{eq:phihestimate}
|\phi-\phi_{h}|_1\lesssim |\phi-I_h\phi|_1 +\| p-I_h^{\delta} p\|+ \|\dd(w- w_{h})\|.
\end{equation}

Next estimate $\| p- p_h\|$. For $\psi\in \Phi_h$ and $s\in V_{k,h}^{\delta,-}\Lambda^{j+3}/\delta V_{k,h}^{\delta,-}\Lambda^{j+4}$,
we have from~\eqref{erroreqn1} and~\eqref{erroreqn3} that
\begin{align*}
b(\psi, s; I_h^{\delta} p- p_h)&=b(\psi, s; I_h^{\delta} p- p) -(\nabla(\phi-\phi_{h}),\nabla\psi) + (\dd(w- w_{h}), \psi).
\end{align*}
Then
\begin{equation*}
b(\psi, s; I_h^{\delta} p- p_h)\lesssim (\| p-I_h^{\delta} p\|+|\phi-\phi_{h}|_1+\|\dd(w- w_{h})\|)(\|\psi\|_1+\|\delta s\|).
\end{equation*}
Employ the discrete inf-sup condition \eqref{BE1} to acquire
\begin{align*}
\| p- p_h\| &\leq \| p-I_h^{\delta} p\|+ \|I_h^{\delta} p- p_h\| \lesssim \| p-I_h^{\delta} p\|+|\phi-\phi_{h}|_1+\|\dd(w- w_{h})\|.  
\end{align*}
Together with \eqref{eq:phihestimate}, we arrive at
\begin{equation*}
|\phi-\phi_{h}|_1+\| p- p_h\|\lesssim |\phi-I_h\phi|_1 +\| p-I_h^{\delta} p\|+ \|\dd(w- w_{h})\|.
\end{equation*}
Finally, we conclude \eqref{eq:phihphestimate} from \eqref{BEqs} and the error estimate \eqref{eq:Ihddelta-estimate} of $I_h^{\delta}$.
\end{proof}

\begin{remark}\rm
The error estimate of $\|p-p_h\|$ for the MINI element method of Stokes equation is suboptimal \cite{ArnoldBrezziFortin1984}. While the error estimate \eqref{eq:phihphestimate} for $\|p-p_h\|$ is optimal when $j\leq d-3$.
\end{remark}

We will use the duality argument to estimate $\|\phi-\phi_{h}\|$. Consider the dual problem:
find $\hat{\phi}\in H_{0}^{1}\Lambda^{j+1}$ and
$\hat{p}\in L^{2}\Lambda^{j+2}/\delta H^*\Lambda^{j+3}$ such that
\begin{equation}
\label{stokesdual}
\left\{
\begin{aligned}
-\Delta\hat{\phi}+\delta\hat{p}&=\phi-\phi_{h} \qquad \mathrm{in}\,\,\Omega, \\
\dd\hat{\phi}&=0 \qquad\qquad \,\,\,\mathrm{in}\,\,\Omega.
\end{aligned}
\right.
\end{equation}%
We assume problem \eqref{stokesdual} has the regularity
\begin{equation}\label{dualregularity}
\|\hat{\phi}\|_2 + \|\hat{p}\|_1\lesssim \|\phi-\phi_{h}\|. 
\end{equation}
When $\Omega$ is a convex polytope in two and three dimensions, the regularity \eqref{dualregularity} holds for $j=d-2, d-1$ (cf. \cite{MazyaRossmann2010,Grisvard1992}). 
By $\dd\hat{\phi}=0$, there exists a $\hat{u}\in H_{0}^{1}\Lambda^{j}$ \cite{CostabelMcIntosh2010} such that 
\begin{equation}\label{uhatregular}
\hat{\phi}=\dd \hat{u},\quad\textrm{and}\quad \|\hat{u}\|_1\lesssim \|\hat{\phi}\|.
\end{equation}

\begin{lemma}
Let $(\phi,  p, 0)\in H_{0}^{1}\Lambda^{j+1}\times L^{2}\Lambda^{j+2}\times H^*\Lambda^{j+3}$ be the solution of problem~\eqref{decoupledformnew2}-\eqref{decoupledformnew3}, and $(\phi_h,  p_h, 0)\in \Phi_h\times V_{k,h}^{\delta,-}\Lambda^{j+2}\times V_{k,h}^{\delta,-}\Lambda^{j+3}$ be the solution of the mixed finite element method \eqref{decoupleMINI2}-\eqref{decoupleMINI3} for $0\leq j\leq d-1$. Assume $\phi\in H^{k+1}\Lambda^{j+1}$, $ p\in H^{k}\Lambda^{j+2}$ and regularity \eqref{dualregularity} holds. We have
\begin{equation}\label{eq:phihphestimateL2}
\|\phi-\phi_{h}\|\lesssim h^{k+1}(|\phi|_{k+1}+| p|_{k}) + h\|\dd(w- w_{h})\|.
\end{equation}
\end{lemma}
\begin{proof}
Applying the integration by parts, we get from the error equation \eqref{erroreqn2}, the error estimate \eqref{BEqs} of $I_h$ and the error estimate \eqref{eq:Ihddelta-estimate} of $I_h^{\delta}$ that
\begin{align*}
\|\phi-\phi_h\|^2&=(\phi-\phi_h, -\Delta\hat{\phi}+\delta\hat{p})=(\nabla(\phi-\phi_h), \nabla\hat{\phi})+(\dd(\phi-\phi_h), \hat{p}) \\
&=(\nabla(\phi-\phi_h), \nabla(\hat{\phi}-I_h\hat{\phi})) +(\dd(\phi-\phi_h), \hat{p}-I_h^{\delta}\hat{p}) \\
&\quad + (\nabla(\phi-\phi_h), \nabla(I_h\hat{\phi})) \\
&\lesssim h|\phi-\phi_h|_1(|\hat{\phi}|_2+|\hat{p}|_1) + (\nabla(\phi-\phi_h), \nabla(I_h\hat{\phi})).
\end{align*}
On the other side, by error equation \eqref{erroreqn1}, \eqref{uhatregular}, error equation \eqref{erroreqnwh1}, \eqref{BEqs} and \eqref{eq:Ihddeltaestimate},
\begin{align*}
(\nabla(\phi-\phi_h), \nabla(I_h\hat{\phi}))&=(\dd(w- w_{h}), I_h\hat{\phi}) - (\dd(I_h\hat{\phi}), p- p_{h}) \\
&=(\dd(w- w_{h}), I_h\hat{\phi}-\hat{\phi}) - (\dd(I_h\hat{\phi}-\hat{\phi}), p- p_{h}) \\
&\quad + (\dd(w- w_{h}), \dd(\hat{u}- I_h^{\dd}\hat{u})) \\
&\lesssim h\|\dd(w- w_{h})\||\hat{\phi}|_1 + h\|p- p_{h}\||\hat{\phi}|_2.
\end{align*}
Therefore \eqref{eq:phihphestimateL2} follows from the last two inequalities, regularity \eqref{dualregularity} and \eqref{eq:phihphestimate}.
\end{proof}

\begin{theorem}\label{thm:MINIerrorestima}
Let $0\leq j\leq d-1$.
Let $(w, 0, \phi, p, 0, u,0)$ be the solution of the decoupled formulation \eqref{decoupleformnew}, and $(w_h, 0, \phi_h, p_h, 0, u_h,0)$ the solution of the decoupled finite element method~\eqref{decoupleMINI}.
Assume $u\in H^{k+2}\Lambda^{j}$.
Then
\begin{align}\label{eq:errorestimatauhH1}
\|\dd(u-u_h)\|&\lesssim h^{k+1}|\dd u|_{k+1}+\|\phi-\phi_{h}\|, \\
\label{eq:errorestimatauhH0}
\|u-u_h\|&\lesssim h^{k+1}(|u|_{k+1}+|\dd u|_{k+1})+\|\phi-\phi_{h}\|.
\end{align}
Further, assume $\phi\in H^{k+1}\Lambda^{j+1}$, $ p\in H^{k}\Lambda^{j+2}$, $w\in H^{k+1}\Lambda^{j}$, and regularity~\eqref{dualregularity} holds, then
\begin{align}
\label{DA}
\|\dd(u-u_h)\|&\lesssim h^{k+1}(|\dd u|_{k+1}+|\phi|_{k+1}+| p|_{k}+|\dd w|_{k}),\\
\label{DA0}
\|u-u_h\|&\lesssim h^{k+1}(|u|_{k+1}+|\dd u|_{k+1}+|\phi|_{k+1}+| p|_{k}+|\dd w|_{k}).
\end{align}
\end{theorem}
\begin{proof}
To estimate $\|\dd(u-u_{h})\|$, subtract \eqref{decoupledformnew4} from \eqref{decoupleMINI4} to get the error equation
\begin{equation*}
  (\dd(u-u_h), \dd\chi)=(\phi-\phi_{h}, \dd\chi) \quad\forall \,\, \chi\in \mathring{V}_{k+1,h}^{\dd,-}\Lambda^{j}.
\end{equation*}
Taking $\chi=I_h^{\dd}u-u_h\in \mathring{V}_{k+1,h}^{\dd,-}\Lambda^{j}$, we get
\begin{align*}
\|\dd(u-u_h)\|^2=(\dd(u-u_h), \dd(u-I_h^{\dd}u))+(\phi-\phi_{h}, \dd(I_h^{\dd}u-u_h)).
\end{align*}
Then
\begin{equation*}
\|\dd(u-u_h)\|\lesssim \|\dd u-I_h^{\dd}(\dd u)\|+\|\phi-\phi_{h}\|,
\end{equation*}
which combined with \eqref{eq:Ihddeltaestimate} gives \eqref{eq:errorestimatauhH1}.
The estimate \eqref{DA} holds from \eqref{eq:phihphestimateL2}-\eqref{eq:errorestimatauhH1} and~\eqref{eq:whphestimate}.

Finally, by employing the triangle inequality, the discrete Poincar\'e inequalities,~\eqref{eq:IhdCDprop} and interpolation estimate \eqref{eq:Ihddelta-estimate}, we conclude estimates \eqref{eq:errorestimatauhH0} and~\eqref{DA0} from estimates \eqref{eq:errorestimatauhH1} and \eqref{DA}, respectively.
\end{proof}

\begin{remark}\rm
Conforming finite element methods of the decoupled formulation~\eqref{decoupleformnew} with $j=0$ and $d=2$, i.e. the biharmonic equation in two dimensions, are analyzed in \cite[Section 4.2]{ChenHuang2018}.
\end{remark}


\section{Numerical results}\label{sec:numeresult}

In this section, we will numerically test the decoupled finite element method~\eqref{decoupleMINIdiag} for $j=0$ in three dimensions, i.e. problem \eqref{4thorderpde} is the biharmonic equation.
Let $\Omega = (0,1)^{3} $ be the unit cube, and the exact solution of biharmonic equation be 
\begin{align*}
u(\boldsymbol{x})=\sin^3(\pi x_1)\sin^3(\pi x_2)\sin^3(\pi x_3).
\end{align*}
The load function $f$ is analytically computed from problem \eqref{4thorderpde}. 
We utilize uniform tetrahedral meshes on $\Omega$.

Numerical errors of the decoupled finite element method \eqref{decoupleMINIdiag} with $k=1$ are shown in Table~\ref{eer1} and Table~\ref{eer2}. From these two tables we can see that $\|u-u_{h}\|=O(h^2)$, $|u-u_{h}|_1=O(h^2)$, $\|\phi-\phi_{h}\|=O(h^2)$ and $|\phi-\phi_{h}|_{1}=O(h)$, which agree with the theoretical estimates \eqref{eq:phihphestimate}, \eqref{eq:phihphestimateL2} and \eqref{DA}-\eqref{DA0}.
\begin{table}[htbp]
\setlength{\abovecaptionskip}{0pt}
\setlength{\belowcaptionskip}{0pt}
\caption{ Errors $\|u-u_{h}\|$ and $|u-u_{h}|_1$ of the decoupled conforming method \eqref{decoupleMINIdiag} with $k=1$.}
\label{eer1}
\begin{center}
\begin{tabular}{ccccccc}
\toprule
$h$   \,  &\,   $\|u-u_{h}\|$ \, & \,   rate    \,  & \,  $|u-u_{h}|_1$ \,  &\,  rate   \\
\midrule
$2^{-2}$  \,  &\,   1.30759E$-$01    \, &  \, $-$     \, &   \, 9.92045E$-$01   \,  &\, $-$ \\
$2^{-3}$ \,  &\,     5.04489E$-$02    \, &  \,  1.3740  \,  &  \, 4.34958E$-$01   \,  &\,  1.1895   \\
$2^{-4}$  \,  &\,     1.42827E$-$02    \, &   \, 1.8206   \, &  \, 1.33687E$-$01   \,  &\,  1.7020  \\
$2^{-5}$   \,  &\,    3.67529E$-$03    \, &   \, 1.9583   \, &  \, 3.52141E$-$02   \,  &\, 1.9246   \\
\bottomrule
\end{tabular}
\end{center}
\end{table}
\begin{table}[htbp]
\setlength{\abovecaptionskip}{0pt}
\setlength{\belowcaptionskip}{0pt}
\caption{ Errors $\|\phi-\phi_{h}\|$ and $|\phi-\phi_{h}|_{1}$ of the decoupled conforming method \eqref{decoupleMINIdiag} with $k=1$.}
\label{eer2}
\begin{center}
\begin{tabular}{ccccc}
\toprule
$h$ \,  &\, $\|\phi-\phi_{h}\|$ \,  &\,   rate  \,  &\,     $|\phi-\phi_{h}|_{1}$ \,  &\,   rate  \\
\midrule
$2^{-2}$  \,  &\,    1.69698E+00    \,  &\,    $-$    \,  &\,    1.10196E+01      \,  &\,  $-$ \\
$2^{-3}$ \,  &\,   7.45455E$-$01   \,  &\,    1.1868    \,  &\,     6.38092E+00    \,  &\,  0.7882 \\
$2^{-4}$  \,  &\,   2.29390E$-$01    \,  &\,   1.7003    \,  &\,    2.83386E+00      \,  &\,  1.1710 \\
$2^{-5}$  \,  &\,    6.04572E$-$02     \,  &\,   1.9238    \,  &\,    1.37843E+00      \,  &\,  1.0397 \\
\bottomrule
\end{tabular}
\end{center}
\end{table}

\bibliographystyle{abbrv}
\bibliography{./refs}

\begin{thebibliography}{10}

\bibitem{AinsworthParker2024}
M.~Ainsworth and C.~Parker.
\newblock Computing {$H^2$}-conforming finite element approximations without
  having to implement {$C^1$}-elements.
\newblock {\em SIAM J. Sci. Comput.}, 46(4):A2398--A2420, 2024.

\bibitem{AinsworthParker2024a}
M.~Ainsworth and C.~Parker.
\newblock Two and three dimensional {$H^2$}-conforming finite element
  approximations without {$C^1$}-elements.
\newblock {\em Comput. Methods Appl. Mech. Engrg.}, 431:Paper No. 117267, 2024.

\bibitem{AnHuangZhang2024}
Q.~An, X.~Huang, and C.~Zhang.
\newblock A decoupled finite element method for the triharmonic equation.
\newblock {\em Appl. Math. Lett.}, 147:Paper No. 108843, 8, 2024.

\bibitem{ArgyrisFriedScharpf1968}
J.~Argyris, I.~Fried, and D.~Scharpf.
\newblock The {TUBA} family of plate elements for the matrix displacement
  method.
\newblock {\em Aero. J. Roy. Aero. Soc.}, 72:701--709, 1968.

\bibitem{ArnoldGuzman2021}
D.~Arnold and J.~Guzm\'{a}n.
\newblock Local {$L^2$}-bounded commuting projections in {FEEC}.
\newblock {\em ESAIM Math. Model. Numer. Anal.}, 55(5):2169--2184, 2021.

\bibitem{Arnold2018}
D.~N. Arnold.
\newblock {\em Finite element exterior calculus}.
\newblock Society for Industrial and Applied Mathematics (SIAM), Philadelphia,
  PA, 2018.

\bibitem{ArnoldBrezziFortin1984}
D.~N. Arnold, F.~Brezzi, and M.~Fortin.
\newblock A stable finite element for the {S}tokes equations.
\newblock {\em Calcolo}, 21(4):337--344, 1984.

\bibitem{ArnoldFalkWinther2006}
D.~N. Arnold, R.~S. Falk, and R.~Winther.
\newblock Finite element exterior calculus, homological techniques, and
  applications.
\newblock {\em Acta Numer.}, 15:1--155, 2006.

\bibitem{ArnoldFalkWinther2010}
D.~N. Arnold, R.~S. Falk, and R.~Winther.
\newblock Finite element exterior calculus: from {H}odge theory to numerical
  stability.
\newblock {\em Bull. Amer. Math. Soc. (N.S.)}, 47(2):281--354, 2010.

\bibitem{BoffiBrezziFortin2013}
D.~Boffi, F.~Brezzi, and M.~Fortin.
\newblock {\em Mixed finite element methods and applications}.
\newblock Springer, Heidelberg, 2013.

\bibitem{BrambleZlamal1970}
J.~H. Bramble and M.~s. Zl\'amal.
\newblock Triangular elements in the finite element method.
\newblock {\em Math. Comp.}, 24:809--820, 1970.

\bibitem{Brenner2003}
S.~C. Brenner.
\newblock Poincar\'{e}-{F}riedrichs inequalities for piecewise {$H^1$}
  functions.
\newblock {\em SIAM J. Numer. Anal.}, 41(1):306--324, 2003.

\bibitem{BrennerCavanaughSung2024}
S.~C. Brenner, C.~Cavanaugh, and L.-y. Sung.
\newblock A {H}odge decomposition finite element method for the quad-curl
  problem on polyhedral domains.
\newblock {\em J. Sci. Comput.}, 100(3):Paper No. 80, 35, 2024.

\bibitem{BrennerSunSung2017}
S.~C. Brenner, J.~Sun, and L.-y. Sung.
\newblock Hodge decomposition methods for a quad-curl problem on planar
  domains.
\newblock {\em J. Sci. Comput.}, 73(2-3):495--513, 2017.

\bibitem{Brezis2011}
H.~Brezis.
\newblock {\em Functional analysis, {S}obolev spaces and partial differential
  equations}.
\newblock Universitext. Springer, New York, 2011.

\bibitem{CaoChenHuang2022}
S.~Cao, L.~Chen, and X.~Huang.
\newblock Error analysis of a decoupled finite element method for quad-curl
  problems.
\newblock {\em J. Sci. Comput.}, 90(1):Paper No. 29, 25, 2022.

\bibitem{ChaconSimakovZocco2007}
L.~Chac{\'o}n, A.~N. Simakov, and A.~Zocco.
\newblock Steady-state properties of driven magnetic reconnection in 2d
  electron magnetohydrodynamics.
\newblock {\em Phys. Rev. Lett.}, 99(23):235001, 2007.

\bibitem{ChenChenHuangWei2024}
C.~Chen, L.~Chen, X.~Huang, and H.~Wei.
\newblock Geometric decomposition and efficient implementation of high order
  face and edge elements.
\newblock {\em Commun. Comput. Phys.}, 35(4):1045--1072, 2024.

\bibitem{ChenHuang2018}
L.~Chen and X.~Huang.
\newblock Decoupling of mixed methods based on generalized {H}elmholtz
  decompositions.
\newblock {\em SIAM J. Numer. Anal.}, 56(5):2796--2825, 2018.

\bibitem{ChenHuang2022b}
L.~Chen and X.~Huang.
\newblock Complexes from complexes: Finite element complexes in three
  dimensions.
\newblock {\em arXiv preprint arXiv:2211.08656}, 2022.

\bibitem{ChenHuang2024a}
L.~Chen and X.~Huang.
\newblock Finite element complexes in two dimensions (in {Chinese}).
\newblock {\em Sci. Sin. Math.}, 54:1--34, 2024.

\bibitem{ChenHuang2024}
L.~Chen and X.~Huang.
\newblock Finite element de {R}ham and {S}tokes complexes in three dimensions.
\newblock {\em Math. Comp.}, 93(345):55--110, 2024.

\bibitem{ChenWuZhongZhou2018}
L.~Chen, Y.~Wu, L.~Zhong, and J.~Zhou.
\newblock Multi{G}rid preconditioners for mixed finite element methods of the
  vector {L}aplacian.
\newblock {\em J. Sci. Comput.}, 77(1):101--128, 2018.

\bibitem{ChenHuangHuang2023}
M.~Chen, J.~Huang, and X.~Huang.
\newblock A robust lower-order mixed finite element method for a strain
  gradient elastic model.
\newblock {\em SIAM J. Numer. Anal.}, 61(5):2237--2260, 2023.

\bibitem{ChristiansenWinther2008}
S.~H. Christiansen and R.~Winther.
\newblock Smoothed projections in finite element exterior calculus.
\newblock {\em Math. Comp.}, 77(262):813--829, 2008.

\bibitem{Ciarlet1978}
P.~G. Ciarlet.
\newblock {\em The finite element method for elliptic problems}.
\newblock North-Holland Publishing Co., Amsterdam-New York-Oxford, 1978.

\bibitem{CostabelMcIntosh2010}
M.~Costabel and A.~McIntosh.
\newblock On {B}ogovski\u\i\ and regularized {P}oincar\'e integral operators
  for de {R}ham complexes on {L}ipschitz domains.
\newblock {\em Math. Z.}, 265(2):297--320, 2010.

\bibitem{FanLiuZhang2019}
R.~Fan, Y.~Liu, and S.~Zhang.
\newblock Mixed schemes for fourth-order {DIV} equations.
\newblock {\em Comput. Methods Appl. Math.}, 19(2):341--357, 2019.

\bibitem{FengZhang2016}
C.~Feng and S.~Zhang.
\newblock Optimal solver for {M}orley element discretization of biharmonic
  equation on shape-regular grids.
\newblock {\em J. Comput. Math.}, 34(2):159--173, 2016.

\bibitem{Gallistl2017}
D.~Gallistl.
\newblock Stable splitting of polyharmonic operators by generalized {S}tokes
  systems.
\newblock {\em Math. Comp.}, 86(308):2555--2577, 2017.

\bibitem{Grisvard1992}
P.~Grisvard.
\newblock {\em Singularities in boundary value problems}.
\newblock Masson, Paris; Springer-Verlag, Berlin, 1992.

\bibitem{Hiptmair2001}
R.~Hiptmair.
\newblock Higher order {Whitney} forms.
\newblock {\em Progress in Electromagnetics Research}, 32:271--299, 2001.

\bibitem{Hiptmair2002}
R.~Hiptmair.
\newblock Finite elements in computational electromagnetism.
\newblock {\em Acta Numer.}, 11:237--339, 2002.

\bibitem{HiptmairXu2007}
R.~Hiptmair and J.~Xu.
\newblock Nodal auxiliary space preconditioning in {${\bf H}({\bf curl})$} and
  {${\bf H}({\rm div})$} spaces.
\newblock {\em SIAM J. Numer. Anal.}, 45(6):2483--2509, 2007.

\bibitem{HuLinWu2024}
J.~Hu, T.~Lin, and Q.~Wu.
\newblock A construction of {$C^r$} conforming finite element spaces in any
  dimension.
\newblock {\em Found. Comput. Math.}, 24(6):1941--1977, 2024.

\bibitem{HuZhangZhang2020}
K.~Hu, Q.~Zhang, and Z.~Zhang.
\newblock Simple curl-curl-conforming finite elements in two dimensions.
\newblock {\em SIAM J. Sci. Comput.}, 42(6):A3859--A3877, 2020.

\bibitem{HuangHuang2011}
J.~Huang and X.~Huang.
\newblock Local and parallel algorithms for fourth order problems discretized
  by the {M}orley-{W}ang-{X}u element method.
\newblock {\em Numer. Math.}, 119(4):667--697, 2011.

\bibitem{HuangHuangXu2012}
J.~Huang, X.~Huang, and J.~Xu.
\newblock An efficient {P}oisson-based solver for biharmonic equations
  discretized by the {M}orley element method.
\newblock {\em Tech. report}, 2012.

\bibitem{Huang2010}
X.~Huang.
\newblock {\em New finite element methods and efficient algorithms for fourth
  order elliptic equations}.
\newblock PhD thesis, Shanghai Jiao Tong University, 2010.

\bibitem{Huang2023}
X.~Huang.
\newblock Nonconforming finite element {S}tokes complexes in three dimensions.
\newblock {\em Sci. China Math.}, 66(8):1879--1902, 2023.

\bibitem{HuangShiWang2021}
X.~Huang, Y.~Shi, and W.~Wang.
\newblock A {M}orley-{W}ang-{X}u element method for a fourth order elliptic
  singular perturbation problem.
\newblock {\em J. Sci. Comput.}, 87(3):Paper No. 84, 24, 2021.

\bibitem{HuangZhang2024}
X.~Huang and C.~Zhang.
\newblock Robust mixed finite element methods for a quad-curl singular
  perturbation problem.
\newblock {\em J. Comput. Appl. Math.}, 451:Paper No. 116117, 19, 2024.

\bibitem{KingsepChukbarYankov1990}
A.~Kingsep, K.~Chukbar, and V.~Yankov.
\newblock Electron magnetohydrodynamics.
\newblock {\em Rev. Plasma Phys.}, 16, 1990.

\bibitem{Ladyzhenskaya1969}
O.~A. Ladyzhenskaya.
\newblock {\em The mathematical theory of viscous incompressible flow}.
\newblock Gordon and Breach Science Publishers, New York, 2nd edition, 1969.

\bibitem{LaiSchumaker2007}
M.-J. Lai and L.~L. Schumaker.
\newblock Trivariate {$C^r$} polynomial macroelements.
\newblock {\em Constr. Approx.}, 26(1):11--28, 2007.

\bibitem{MazyaRossmann2010}
V.~Maz'ya and J.~Rossmann.
\newblock {\em Elliptic equations in polyhedral domains}, volume 162 of {\em
  Mathematical Surveys and Monographs}.
\newblock American Mathematical Society, Providence, RI, 2010.

\bibitem{MindlinEshel1968}
R.~Mindlin and N.~Eshel.
\newblock On first strain-gradient theories in linear elasticity.
\newblock {\em International Journal of Solids and Structures}, 4(1):109--124,
  1968.

\bibitem{Mindlin1964}
R.~D. Mindlin.
\newblock Micro-structure in linear elasticity.
\newblock {\em Arch. Rational Mech. Anal.}, 16:51--78, 1964.

\bibitem{Nedelec1980}
J.-C. N\'{e}d\'{e}lec.
\newblock Mixed finite elements in {${\bf R}\sp{3}$}.
\newblock {\em Numer. Math.}, 35(3):315--341, 1980.

\bibitem{Nicolaides1972}
R.~A. Nicolaides.
\newblock On a class of finite elements generated by {L}agrange interpolation.
\newblock {\em SIAM J. Numer. Anal.}, 9:435--445, 1972.

\bibitem{PasciakZhao2002}
J.~E. Pasciak and J.~Zhao.
\newblock Overlapping {S}chwarz methods in {$H$}(curl) on polyhedral domains.
\newblock {\em J. Numer. Math.}, 10(3):221--234, 2002.

\bibitem{Reddy2006}
J.~N. Reddy.
\newblock {\em Theory and Analysis of Elastic Plates and Shells}.
\newblock CRC Press, Boca Raton, 2nd edition, 2006.

\bibitem{Schedensack2016}
M.~Schedensack.
\newblock A new discretization for {$m$}th-{L}aplace equations with arbitrary
  polynomial degrees.
\newblock {\em SIAM J. Numer. Anal.}, 54(4):2138--2162, 2016.

\bibitem{Sun2016}
J.~Sun.
\newblock A mixed {FEM} for the quad-curl eigenvalue problem.
\newblock {\em Numer. Math.}, 132(1):185--200, 2016.

\bibitem{SunZhangZhang2019}
J.~Sun, Q.~Zhang, and Z.~Zhang.
\newblock A curl-conforming weak {G}alerkin method for the quad-curl problem.
\newblock {\em BIT}, 59(4):1093--1114, 2019.

\bibitem{Zenisek1970}
A.~\v{Z}en\'{\i}\v{s}ek.
\newblock Interpolation polynomials on the triangle.
\newblock {\em Numer. Math.}, 15:283--296, 1970.

\bibitem{Zenisek1974}
A.~\v{Z}en\'{\i}\v{s}ek.
\newblock Tetrahedral finite {$C^{(m)}$}-elements.
\newblock {\em Acta Univ. Carolinae---Math. et Phys.}, 15(1-2):189--193, 1974.

\bibitem{Wang2001}
M.~Wang.
\newblock On the necessity and sufficiency of the patch test for convergence of
  nonconforming finite elements.
\newblock {\em SIAM J. Numer. Anal.}, 39(2):363--384, 2001.

\bibitem{ZhangWangZhang2019}
Q.~Zhang, L.~Wang, and Z.~Zhang.
\newblock {$H({\rm curl}^2)$}-conforming finite elements in 2 dimensions and
  applications to the quad-curl problem.
\newblock {\em SIAM J. Sci. Comput.}, 41(3):A1527--A1547, 2019.

\bibitem{ZhangZhang2020}
Q.~Zhang and Z.~Zhang.
\newblock A family of curl-curl conforming finite elements on tetrahedral
  meshes.
\newblock {\em CSIAM Trans. Appl. Math.}, 1(4):639--663, 2020.

\bibitem{ZhangZhang2022}
Q.~Zhang and Z.~Zhang.
\newblock Three families of grad div-conforming finite elements.
\newblock {\em Numer. Math.}, 152(3):701--724, 2022.

\bibitem{Zhang2009}
S.~Zhang.
\newblock A family of 3{D} continuously differentiable finite elements on
  tetrahedral grids.
\newblock {\em Appl. Numer. Math.}, 59(1):219--233, 2009.

\bibitem{Zhang2016}
S.~Zhang.
\newblock A family of differentiable finite elements on simplicial grids in
  four space dimensions.
\newblock {\em Math. Numer. Sin.}, 38(3):309--324, 2016.

\bibitem{Zhang2018}
S.~Zhang.
\newblock Regular decomposition and a framework of order reduced methods for
  fourth order problems.
\newblock {\em Numer. Math.}, 138(1):241--271, 2018.

\bibitem{Zulehner2011}
W.~Zulehner.
\newblock Nonstandard norms and robust estimates for saddle point problems.
\newblock {\em SIAM J. Matrix Anal. Appl.}, 32(2):536--560, 2011.

\end{thebibliography}
\end{document}